\documentclass[10pt]{amsart}
\usepackage{amsfonts}
\usepackage{amssymb}
\usepackage[mathscr]{euscript}
\usepackage[all]{xy}
\usepackage[dvips]{graphics}
\usepackage{amsmath,amssymb,epsfig}

\newtheorem{theorem}{Theorem}
\newtheorem{lemma}[theorem]{Lemma}

\newtheorem{corollary}[theorem]{Corollary}

\title{On the Combinatorics of Smoothing}
\author{Micah W. Chrisman}
\begin{document}
\subjclass[2000]{57M25,57M27}
\begin{abstract}
Many invariants of knots rely upon smoothing the knot at its crossings.  To compute them, it is necessary to know how to count the number of connected components the knot diagram is broken into after the smoothing.  In this paper, it is shown how to use a modification of a theorem of Zulli together with a modification of the spectral theory of graphs to approach such problems systematically.  We give an application to counting subdiagrams of pretzel knots which have one component after oriented and unoriented smoothings. 
\end{abstract}
\maketitle
\section{Introduction}
In \cite{MR1341816}, Zulli gave a beautiful method by which to compute the number of state curves in any state. By a \emph{state}, we mean a choice of \emph{oriented} or \emph{unoriented} smoothing at each crossing in an oriented virtual knot diagram.
\[
\begin{array}{|ccc|} \hline
\begin{array}{c} \underline{\text{Unoriented:}} \\ d_i=1 \end{array} & \begin{array}{c} \scalebox{.25}{\psfig{figure=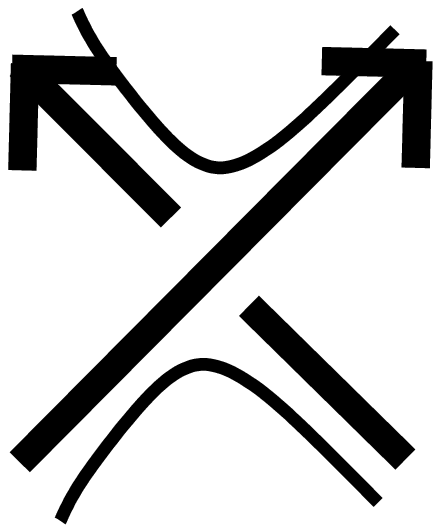}} \end{array} & \begin{array}{c} \scalebox{.25}{\psfig{figure=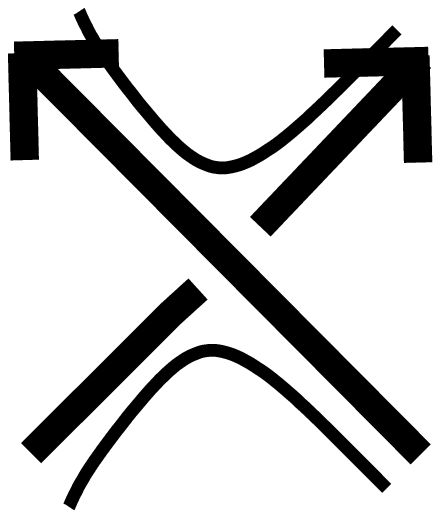}} \end{array} \\
\begin{array}{c} \underline{\text{Oriented:}} \\ d_i=0 \end{array} & \begin{array}{c} \scalebox{.25}{\psfig{figure=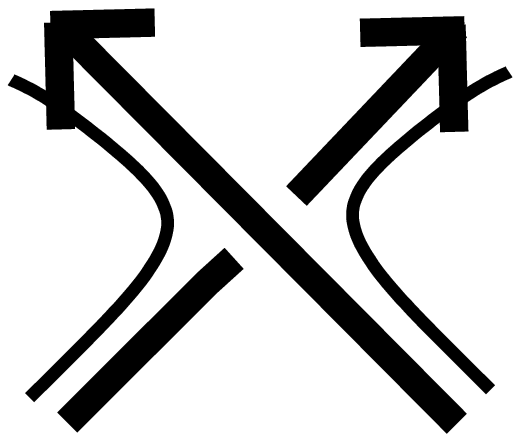}} \end{array} & \begin{array}{c} \scalebox{.25}{\psfig{figure=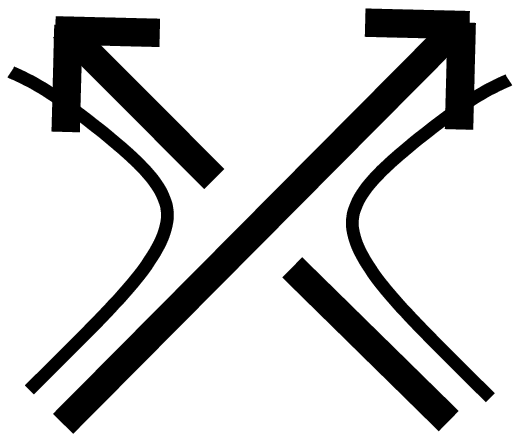}} \end{array}\\ \hline
\multicolumn{3}{|c|}{\Delta=\text{diag}(d_1,\ldots,d_n)} \\ \hline
\end{array}
\]
Using this convention (which is different than the convention in \cite{MR1341816}), Zulli's counting principle may be described as follows.  Let $D$ be the Gauss diagram of a virtual knot $K$ and $G$ the intersection graph of $D$. Let $A$ be the adjacency matrix of the intersection graph (see Figure \ref{zulli_fig}). Labelling the vertices arbitrarily from $1$ to $n$, we set $d_i=0$ if the $i$-th crossing has an oriented smoothing and $d_i=1$ if the $i$-th crossing has the unoriented smoothing. Let $\Delta=\text{diag}(d_1,\ldots,d_n)$, the diagonal matrix with these entries. For any state $S$, let $\#_D(S)$ denote the number of state curves in $K$ after smoothing.
\newline
\newline
\centerline{
\fbox{\parbox{4.2in}{\underline{\textbf{Zulli's Loop Counting Principle (ZLCP):}} The number of state curves in the smoothing $S=(d_1,\ldots,d_n)$ is: 
\newline
\centerline{
$\#_D(S)=\text{nullity}_{\mathbb{Z}_2}(A+\Delta)+1$}}}
}
\hspace{1cm}
\newline
\newline     
In the present paper, we introduce a \emph{refined loop counting principle}(RLCP). The aim of the refinement is to facilitate the computation of combinatorial invariants of virtual knots.  The principle is especially useful when applied to infinite families of knots. Such infinite families arise frequently in virtual knot theory \cite{C1,C4pap}.

While Zulli's loop counting principle works for an arbitrarily labelled intersection graph, the refined loop counting principle requires the introduction of \emph{linearly ordered graphs}. A linearly ordered graph is a simple graph having vertices labelled $1$ to $n$ and edges directed $u \to v$ whenever $u <v$.  A Gauss diagram with a base point gives a canonical ordering of the arrows, so its intersection graph $G$ becomes a linearly ordered graph $\vec{G}$ relative to this ordering. Let $\vec{A}_{\vec{G}}$ denote the \emph{skew-adjacency matrix} of $\vec{G}$. With our choice for the directed edges, the skew-adjacency matrix is non-negative above the diagonal and non-positive below the diagonal.  The diagonal itself consists only of zeros.

\begin{figure}
\[
\begin{array}{|ccc|} \hline
\begin{array}{c}\scalebox{.5}{\psfig{figure=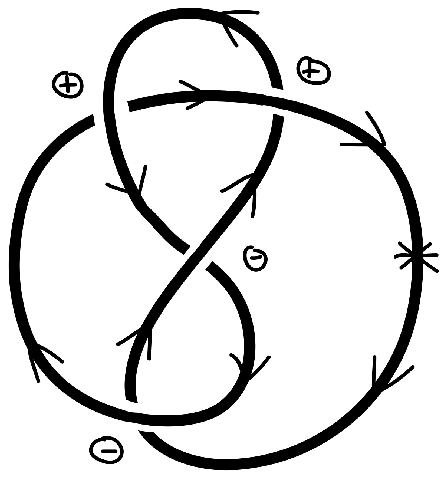}} \end{array} &\begin{array}{c}\psfig{figure=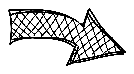} \end{array} & \begin{array}{c}\scalebox{.5}{\psfig{figure=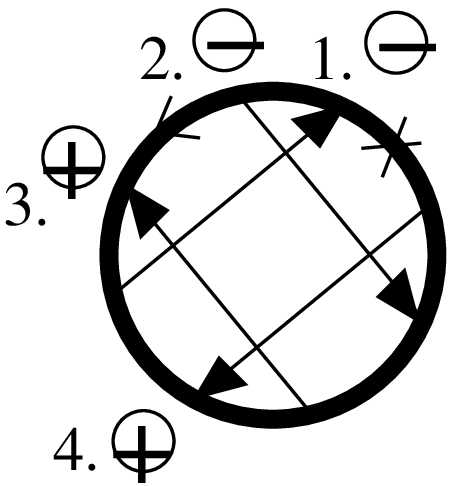}} \end{array}\\
& & \begin{array}{c}\psfig{figure=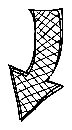} \end{array} \\
\left[\begin{array}{cccc}0 & 1 & 1 & 0 \\ 1 & 0 & 0 & 1 \\ 1 & 0 & 0 & 1 \\ 0& 1 & 1 & 0 \end{array}  \right]& \begin{array}{c}\psfig{figure=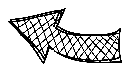} \end{array} & \begin{array}{c}\scalebox{.5}{\psfig{figure=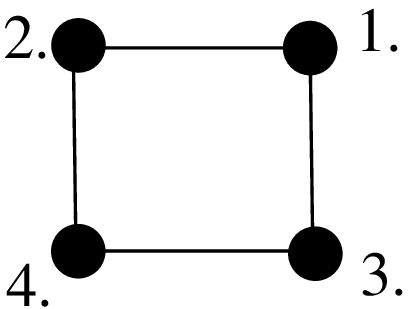}} \end{array}\\ \hline
\end{array}
\]
\caption{A knot and it's Gauss diagram, intersection graph, and adjacency matrix.} \label{zulli_fig}
\end{figure}

The RLCP will also apply to smoothing along any subset of crossings.  Let $S_o$ be the set of oriented smoothings, $S_u$ the set of unoriented smoothings, and $S_{\emptyset}$ the set of crossings which are not smoothed. We call $S=(S_o,S_u,S_{\emptyset})$ a \emph{partial smoothing} of the virtual knot.  If $D$ is a Gauss diagram of the virtual knot diagram, let $D_{\emptyset}$ denote $D$ with all arrows corresponding to crossings in $S_{\emptyset}$ erased. Let $O$ be the choice of smoothing where all arrows of a diagram are given the oriented smoothing.

For any square matrix $A$, Let $m_0(A)$ denote the multiplicity of zero as a root of the characteristic polynomial of $A$.  The refined loop counting principle may then be stated as:
\newline
\newline
\centerline{
\fbox{\parbox{4.2in}{\underline{\textbf{Refined Loop Counting Principle (RLCP):}} Let $S=(S_o,S_u,S_{\emptyset})$ be a partial state. Let $\vec{G}_{\emptyset}$ be the linearly ordered graph of $D_{\emptyset}$. 
\newline
\begin{tabular}{|c|}  \hline \sffamily A \\ \hline \end{tabular} If $S_u=\emptyset$, then $\#_D(S)=m_0(\vec{A}_{\vec{G}_{\emptyset}})+1$.
\newline
\begin{tabular}{|c|}  \hline \sffamily B \\ \hline \end{tabular} If $j \in S_u \ne \emptyset$, there are Gauss diagrams $D_j^f(S)$ and $D_j^s(S)$, called  \emph{double covers} of $D$ such that $\#_{D_j^f(S)}(O)= 2\cdot \#_D(S)=\#_{D_j^s(S)}(O)$ (see Section \ref{dblcover_sec}) .
\newline
\begin{tabular}{|c|}  \hline \sffamily C \\ \hline \end{tabular} $m_0(\vec{A}_{\vec{G}})$ may be computed directly from simpler linear ordered graphs with known characteristic polynomials.
}}
}
\hspace{1cm}
\newline
Part (A) of the RLCP is established along lines similar to the ZLCP.  In our case, the major difference lies in computing homology over $\mathbb{Q}$ rather than $\mathbb{Z}_2$.  Part (B) of the RLCP says that any loop counting problem can be reduced to loop counting of the all oriented state.  To do this, we construct a Gauss diagram from a topological double cover of a non-orientable surface $\Sigma_D(S)$ associated to $D$ and $S$ (see Section \ref{dblcover_sec}).  Part (C) is established by considering a modification of \emph{spectral graph theory}\cite{MR2571608,MR1440854} for linearly ordered graphs.  Indeed, we show how to compute characteristic polynomials of linear ordered graphs from mirror images, joins, coalescence, and adding an edge. The results are somewhat different than in the standard symmetric case. It is hoped that the RLCP will be viewed as both an interesting application of spectral graph theory and as a useful tool for combinatorial knot theorists.

We should pause for a moment to consider why the RLCP is necessary. In \cite{MR1341816}, Zulli introduced the ZLCP in order to compute the Jones polynomial of knot. Besides the Jones polynomial, there are many virtual knot invariants which require the use of loop counting.  In \cite{CKR}, it was shown that there exists a Gauss diagram formula for the Conway polynomial.  The diagrams in the formula may be described as those diagrams which have one component after applying the all oriented smoothing and are \emph{ascending} (first passage of an arrow is in the direction of the arrow). This invariant also extends to long virtual knots \cite{ChP}. In \cite{C4pap}, it was shown that this extension to long virtual knots can be used to define infinitely many inequivalent extensions of the Conway polynomial to long virtual knots, all of which satisfy the same skein relation. Also, loop counting can be used to generalize many other knot invariants \cite{BP}. The RLCP is thus a useful method to aid in the computation of these new invariants, especially on infinite families of knots (like torus knots, pretzel knots, twist sequences, fractional twist sequences, etc.).

The author was additionally inspired by the impressive accomplishments of Zulli and Traldi \cite{MR2597243}, Traldi \cite{MR2646647}, Manturov and Ilyutko \cite{IM1}, and Ilyutko, Manturov, and Nikonov\cite{INM}. In these works, the $\mathbb{Z}_2$-nullity of the adjacency matrix is used to extend the ideas of knot theory to graph theory. Traldi and Zulli have introduced the idea of loop interlacement graphs while Ilyutko, Manturov and Nikonov have developed the concept of graph-links. It is hoped that the techniques presented here will lead to new insights for graph-links and loop interlacement graphs.

The organization of this paper is as follows. In Section \ref{loopcounting}, we define linearly ordered graphs and establish parts (A) and (B) of the RLCP. In Section \ref{skewspec}, we develop a version of spectral graph theory which applies to linearly ordered graphs and thus establish part (C) of the RLCP. In Section \ref{pretzels} we illustrate the loop counting principle by counting one-component subdiagrams of pretzel knots for different kinds of smoothings. We conclude in Section \ref{problems} with some open problems and questions.
\newline
\newline
\textbf{Acknowledgements:} The author is very grateful to the anonymous reviewer who found an error in the statement of Theorem \ref{pqrodd}. The author is also grateful for the interest of C. Frohman, A. Lowrance, and M. Saito in this work.

\section{Loop Counting on Gauss Diagrams} \label{loopcounting}

\subsection{Gauss Diagrams, Intersection Graphs, and Linearly Ordered Graphs} The reader will be assumed to be familiar with the notion of a virtual knot diagram (see, for example, \cite{KaV}). We may consider any oriented virtual knot diagram as an immersion $S^1 \to \mathbb{R}^2$ where the double points are given as either classical crossings (with local orientation $\oplus$ or $\ominus$) or virtual crossings. The other points are regular points of the immersion. We choose a basepoint at some regular point and leave it fixed throughout. The basepoint is denoted by $*$.  The Gauss diagram is obtained by connecting the pre-images in $S^1$ of each classical double point by a line segment in $\mathbb{R}^2$. The line segments are directed from the over-crossing arc to the under-crossing arc.  We also decorate each arrow with the local orientation of the corresponding crossing: $\oplus$ or $\ominus$. The pre-image of the basepoint is also marked with an $*$. This is illustrated in Figure \ref{vknottogauss}. 

\begin{figure}[h]
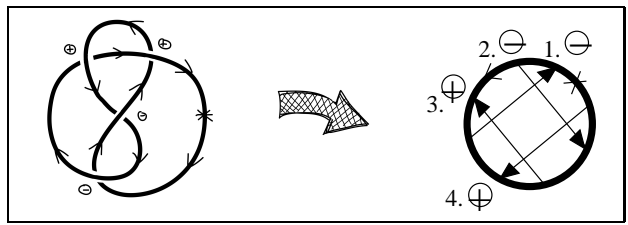

\[
\begin{array}{|ccc|} \hline
\begin{array}{c}\scalebox{.5}{\psfig{figure=present_example_knot.eps}} \end{array} &\begin{array}{c}\psfig{figure=present_arrow_right.eps} \end{array} & \begin{array}{c}\scalebox{.5}{\psfig{figure=present_example_gauss.eps}} \end{array} \\ \hline
\end{array}
\]
\caption{A virtual knot with a basepoint and its Gauss diagram}\label{vknottogauss}
\end{figure}

For the purposes of loop counting, the directions and signs of the arrows are irrelevant as long as it is known whether the smoothing is oriented or unoriented. We will thus only consider the underlying chord diagram of the Gauss diagram. Let $\mathscr{D}(S^1,*)$ denote the collection of chord diagrams on $S^1$ with a basepoint $*$. 

If $D \in \mathscr{D}(S^1,*)$, then the chords of $D$ have a \emph{canonical ordering}. The chord having the first endpoint encountered while travelling CCW from $*$ is labelled one. Delete the chord labelled one. The chord with the first endpoint in this diagram is labelled 2 in the original.  This process is repeated until all chords are labelled.

For any $D \in \mathscr{D}(S^1,*)$, we define the intersection graph $G_D$ as follows \cite{cdbook}. The vertices of $G_D$ are in one-to-one correspondence with the chords of $D$. Two vertices $u$, $v$ are connected by an edge if their corresponding chords in the Gauss diagram intersect. Each vertex of $G_D$ is labelled with the canonical ordering of the corresponding chord.  For $v \in G_D$, let $l(v)$ denote the label of $v$. If $i$ is a chord in a diagram $D \in \mathscr{D}(S^1,*)$, then  the \emph{degree of} $i$ is the degree of the corresponding vertex $v_i$ in $G_D$. 

The intersection graph may be directed according to the canonical ordering. If $u,v \in V(\vec{G}_D)$ and $u \sim v$, then the edge is directed towards $v$ if $l(u) <l(v)$ and towards $u$ if $l(v)<l(u)$. All intersection graphs $G_D$ for $D \in \mathscr{D}(S^1,*)$ are assumed to be labelled and directed according to the canonical ordering. An illustration for the virtual knot in Figure \ref{zulli_fig} is given in Figure \ref{lin_ord_fig}.

Many of our results will hold not just for directed intersection graphs of chord diagrams, but also for a class of directed graphs which we will call \emph{linearly ordered graphs}: Let $G$ be a graph with $n$ vertices labelled $1,\ldots,n$. If $v$ is a vertex of $G$, let $l(v)$ denote its label. Suppose that the edges of $G$ are directed so that $u \to v$ if and only if $l(u)<l(v)$.  A graph whose vertices have been so labelled and directed will be called \emph{linearly ordered}. A linearly ordered graph will be denoted $\vec{G}$.  

The \emph{skew-adjacency matrix} of $\vec{G}$ is defined to be the adjacency matrix of the directed graph $\vec{G}$.  In particular, the columns of the matrix are ordered by the labels of the vertices.  The $a_{ij}$ entry is 1 if $i \sim j$ and $l(i)<l(j)$, 0 if $i \not\sim j$, and $-1$ if $i \sim j$ and $l(i)>l(j)$. By construction, all of the elements of the skew-adjacency matrix above the diagonal are non-negative and all of the elements below the diagonal are non-positive (compare with \cite{MR2571608}). The skew-adjacency matrix will be denoted $\vec{A}_{\vec{G}}$. This is illustrated in Figure \ref{lin_ord_fig}. The characteristic polynomial $\det(xI-\vec{A}_{\vec{G}})$ of $\vec{A}_{\vec{G}}$ will be denoted by $\vec{P}_{\vec{G}}(x)$. 

\begin{figure}
\[
\begin{array}{|ccc|} \hline
\multicolumn{3}{|l|}{\underline{\text{Linearly Ordered Graphs}:}} \\
\vec{G}=\begin{array}{c}\scalebox{.5}{\psfig{figure=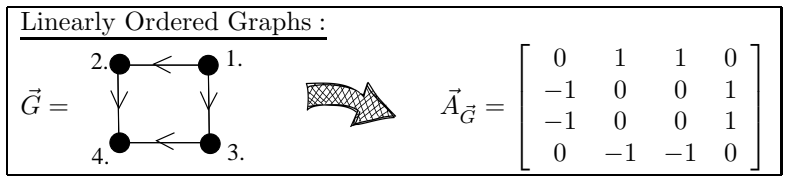}} \end{array} & \begin{array}{c}\psfig{figure=present_arrow_right.eps} \end{array} &
\vec{A}_{\vec{G}}=\left[\begin{array}{cccc}0 & 1 & 1 & 0 \\ -1 & 0 & 0 & 1 \\ -1 & 0 & 0 & 1 \\ 0 & -1 & -1 & 0 \end{array}  \right] \\ \hline
\end{array}
\]
\caption{The Linearly Ordered Graph and Skew-Adjacency Matrix of the knot in Figure \ref{zulli_fig}.} \label{lin_ord_fig}
\end{figure}   

\subsection{States, Smoothings and Bands} A \emph{partial state} $S$ of a virtual knot $K$ is a choice at each crossing of an oriented smoothing, an unoriented smoothing, or no smoothing. Let $S_o$ be the set of crossings of $K$ with the oriented smoothing, $S_u$ the set of crossings with the unoriented smoothing, and $S_{\emptyset}$ the set of crossings which are not smoothed. We write $S=(S_o,S_u,S_{\emptyset})$. Let $O$ denote the partial state where all crossings are given the oriented smoothing. By a \emph{no-unoriented smoothing}, we mean any partial smoothing with $S_u=\emptyset$.

\begin{figure}[h]
\[
\begin{array}{cc}
\scalebox{.5}{\psfig{figure=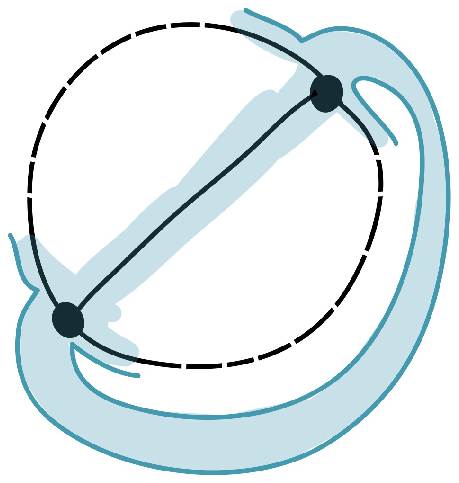}} & \scalebox{.5}{\psfig{figure=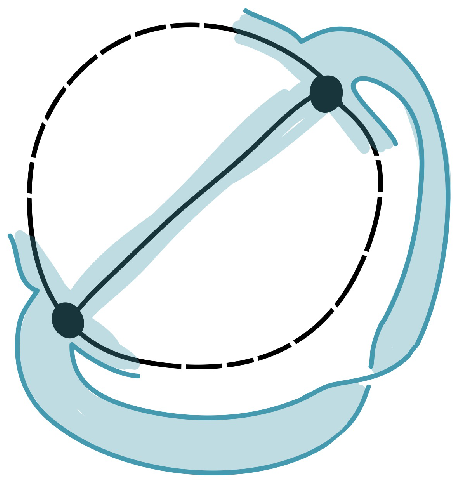}} \\
\end{array}
\]
\caption{Adding an untwisted band (left) and a half-twisted band (right).} \label{addaband}
\end{figure}

Now suppose $D \in \mathscr{D}(S^1,*)$. A \emph{partial state} $S=(S_o,S_u,S_{\emptyset})$ is a partition of the chords of $D$ into three sets where at most two of $S_o$, $S_u$, and $S_{\emptyset}$ are empty. The \emph{partial smoothing} is a compact surface $\Sigma_D(S)$ defined as follows. First, the circle of $D$ is considered as the boundary of a disk $D^2$. Let $D_{\emptyset}$ denote the diagram obtained from $D$ by deleting every chord in $S_{\emptyset}$. For every $i \in S_o$, we attach an untwisted band to the endpoints of $i$ in $D_{\emptyset}$ as in the left hand side of Figure \ref{addaband}. For every $i \in S_u$, we attach a band with a half-twist as in the right hand side of Figure \ref{addaband}. The resulting surface is $\Sigma_D(S)$ (see Figure \ref{defofsigma}). 

\begin{figure}[h]
\[
\begin{array}{cc}
\begin{array}{c}\scalebox{.5}{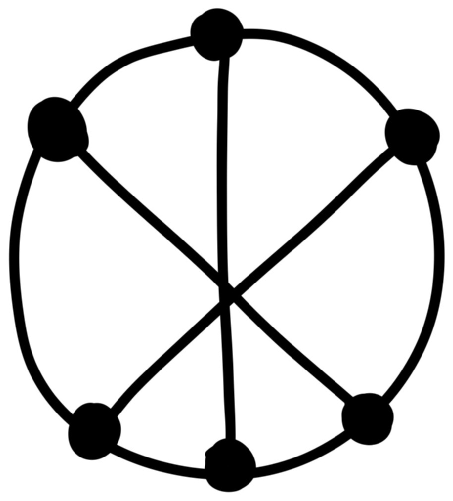}\end{array} &  \begin{array}{c}\scalebox{.5}{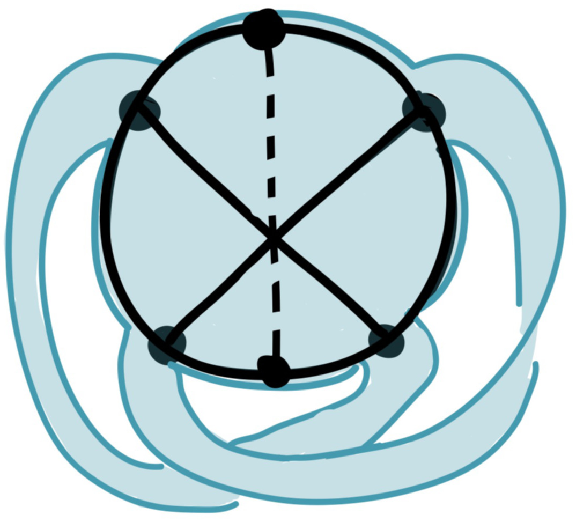}\end{array} \\
\end{array}
\]
\caption{A chord diagram $D$ (left) and partial smoothing $\Sigma_D(S)$ (right).} \label{defofsigma}
\end{figure}
\hspace{1cm}
\newline
\newline
\textbf{Definitions:} The number of boundary components of $\Sigma_D(S)$ will be denoted $\#_D(S)$. The number of closed curves of a partial smoothing $S$ of virtual knot $K$ is denoted $\#(K|S)$.

\begin{lemma} Let $S$ be a partial state of a based virtual knot diagram $K$. Let $D$ be the chord diagram of a Gauss diagram of $K$ with the canonical labelling (hence, the partial smoothing $S$ gives a partition of the set of chords of $D$). Then $\#(K|S)=\#_D(S)$. 
\end{lemma}

\begin{proof} Note that an oriented smoothing at a crossing creates a two component link. An unoriented smoothing at a crossing returns a one component knot. Also, note that making a classical crossing of $K$ into a virtual crossing corresponds to deleting the chord of the crossing in $D$. The procedure for adding bands in the definition of $\Sigma_D(S)$ models this procedure exactly. Hence, the number of boundary components of $\Sigma_D(S)$ is exactly the number of immersed curves in the partial state.
\end{proof}

It will be useful in the sequel to refer to the endpoints of the chords and the corners of each band.  The endpoint of the chord $i$ occurring first when travelling CCW from the basepoint is labelled $a_i$.  The other endpoint is labelled $b_i$. The two corners of side of the $i$-th band containing $a_i$ are labelled $a_i'$ and $a_i''$, where $a_i'$ comes before $a_i$ when travelling from the basepoint and $a_i''$ comes after $a_i$ when travelling from the basepoint. Similarly, the two corners of the side of the band containing $b_i$ are labelled $b_i'$ and $b_i''$, where $b_i'$ is encountered before $b_i$ and $b_i''$ is encountered after $b_i$.  For an illustration, see Figure \ref{aprimedef}. Note that the notation is independent of the choice of smoothing.

\begin{figure}[h]
\scalebox{.5}{
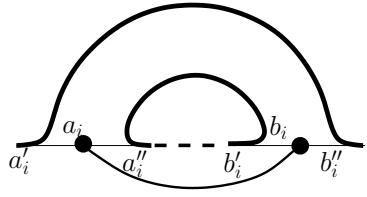}
\caption{The labelling of the endpoints of a chord and the corners of its band.} \label{aprimedef}
\end{figure}

\subsection{The No-Unoriented Smoothing: Zulli's Theorem over $\mathbb{Q}$} \label{zulthm} In this section, we establish the validity of part (A) of the RLCP. First, it is shown that the number of partial state curves can be computed from the rational nullity of the skew-adjacency matrix.  It is then proved that this is identically the multiplicity of zero as a root of the characteristic polynomial of the skew-adjacency matrix. 

\begin{theorem}\label{zullthmstate}[Zulli's Theorem over $\mathbb{Q}$] Let $D \in \mathscr{D}(S^1,*)$. Let $S=(S_o,S_u,S_{\emptyset})$ be a partial state of $D$ with $S_u=\emptyset$. Let $\vec{G}_{\emptyset}$ be the canonical linearly ordered graph of $D_{\emptyset}$. Then:
\[
\#_D(S)=\text{nullity}_{\mathbb{Q}}(\vec{A}_{\vec{G}_{\emptyset}}^t)+1=\text{nullity}_{\mathbb{Q}}(\vec{A}_{\vec{G}_{\emptyset}})+1=\dim_{\mathbb{Q}}H_1(\partial\Sigma_D(S);\mathbb{Q}).
\] 
\end{theorem}
\begin{proof} Let $n$ be the number of chords of $D_{\emptyset}$. Let $j:(\Sigma_D(S),\emptyset) \to (\Sigma_D(S),\partial \Sigma_D(S))$ denote the inclusion of pairs. By \cite{MR1341816}, it is sufficient to show that the skew-adjacency matrix $\vec{A}_{\vec{G}_{\emptyset}}$ represents the map:

\begin{eqnarray*}
j_*:H_1(\Sigma_D(S);\mathbb{Q}) & \to & H_1(\Sigma_D(S),\partial \Sigma_D(S);\mathbb{Q}) \\
                         & \cong & H^1(\Sigma_D(S) ; \mathbb{Q}) \,\,\text{(Lefschetz duality)}\\
                         & \cong & \text{Hom}_{\mathbb{Q}}(H_1(\Sigma_D(S);\mathbb{Q}),\mathbb{Q}). 
\end{eqnarray*}

Consider the band $I \times I$ attached along the arcs $a_k'a_k''$ and $b_k'b_k''$. Let $c_k$ be a $1$-simplex in the band with endpoints $a_k$ and $b_k$ corresponding to the central arc of the band. The central arc of the band is the homeomorphic image of $\{1/2\} \times [0,1]$ in $\Sigma_D(S)$. This $1$-simplex is oriented from $b_i$ to $a_i$. Let $c_k'$ denote the $1$-simplex of the chord labelled $k$. We orient $c_k'$ from $a_i$ to $b_i$.  Then  $c_k+c_k'$ is a $1$-cycle in $Z(\Sigma_D(S);\mathbb{Q})$ (see Figure \ref{defofi}). We will denote this $1$-cycle of the chord labelled $k$ by $k$.

The set of $1$-cycles $k$ just defined is a basis for $H_1(\Sigma_D(S);\mathbb{Q})$. We will denote this basis by $\mathscr{B}=\{1,2,\ldots,n\}$. A choice of basis for $H_1^*(\Sigma_D(S);\mathbb{Q})$ can be made canonical as follows. Start from the basepoint and move in the direction of orientation. The arc directed from $a_i'$ to $a_i''$ can be considered as a 1-cycle in $Z(\Sigma_D(S),\partial \Sigma_D(S);\mathbb{Q})$. Denote this class by $i^*$. By the proof of Alexander-Poincar\'{e} duality (see \cite{MR1402473}, Theorem 5.3.13), the ordered set $\{1^*,\ldots,n^* \}$ can be taken as an ordered basis of $H_1(\Sigma_D(S),\partial \Sigma_D(S);\mathbb{Q})$.

Following Zulli \cite{MR1341816}, we replace the class $i \in \mathscr{B}$ by a homologous class $i'$ such that $j_*(i)=j_*(i')$ (drawn red in Figure \ref{defofi}). The class of $i'$ is the class of the $1$-simplex from $a_i''$ to $b_i'$ together with the arc from $b_i'$ to $a_i''$ along the boundary of the $i$-th band.
\begin{figure}
\scalebox{.5}{\psfig{figure=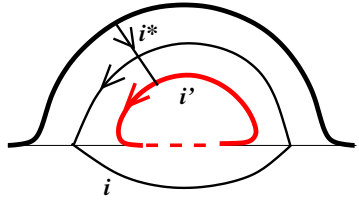}}
\caption{The 1-cycles $i$, $i'$, and $i^*$.} \label{defofi}
\end{figure}
With this notation, it is clear that $j_*(i)=j_*(i')$.  Moreover, it is seen that $j_*(i')$ modulo the boundary is the $i$-th \emph{row} of the skew-adjacency matrix. Indeed, if a chord labelled $k$ intersects the chord labelled $i$ and $i<k$, then the contribution to the linear combination is $+1$.  On the other hand, if $k<i$, then the contribution is $-1$. As the surface is orientable, the remainder of the argument in \cite{MR1341816} follows without alteration to give:
\[
\dim_{\mathbb{Q}}(\ker(j_*))+1=\dim_{\mathbb{Q}}(H_1(\partial \Sigma_D(S),\mathbb{Q})).
\]
This completes the proof.
\end{proof}

To establish (A) of the RLCP, it remains to show that the number of loops in the all-oriented state can be computed directly from the characteristic polynomial. First we must identify the eigenvectors in the $\lambda=0$ eigenspace. We use the notation of Theorem \ref{zullthmstate}. Let $\mathscr{E}(0)$ denote the $\lambda=0$ eigenspace of $\vec{A}_{\vec{G_{\emptyset}}}$. A spanning set of eigenvectors for $\mathscr{E}(0)$ can be determined as follows. Let $\mathscr{S}_D$ denote the set of boundary components of $\Sigma_D(S)$. Let $\iota(C)$ be the set of all arcs $a_i'b_i''$ and $a_i''b_i'$ on the $i$-th band which the component $C \in \mathscr{S}_D$ contains. Let $c \in \iota(C)$. If $c=a_i'b_i''$ for some $i$, define $\sigma(c)=1$. If $c=a_i''b_i'$ for some $i$, define $\sigma(c)=-1$. If $c \in \iota(C)$ is on the $i$-th band, let $e_c$ denote the vector of length $n$($=$ number of chords) having a 1 in the $i$-th position and zero elsewhere. 

\begin{corollary} \label{zeroeigen} Given the notation as above, the following statements hold.
\begin{enumerate}
\item For all $C \in \mathscr{S}_D$, the element $\theta_C=\sum_{c \in \iota(C)} \sigma(c) e_c \in \ker(\vec{A}_{\vec{G}_{\emptyset}})$.
\item The set $\beta(D)=\{\theta_C: C \in \mathscr{S}_D\}$ is a spanning set for the $\lambda=0$ eigenspace of $\vec{A}_{\vec{G}_{\emptyset}}$ and is hence linearly dependent.
\item If $\beta(D) \ne \{\vec{0}\}$, there is an element of $\beta(D)$ which can be removed from $\beta(D)$ to give a basis for $\mathscr{E}(0)$.
\item The dimension of the $\lambda=0$ eigenspace is equal to the algebraic multiplicity of $0$ as a root of the characteristic polynomial.
\end{enumerate} 
\end{corollary}
\begin{proof} For the first claim, we note that with our choice of signs, $\theta_C$ is just the homology class of $\pm [C]$ written in terms of the basis $\mathscr{B}$ from the proof of Theorem \ref{zullthmstate}. Hence $\theta_C$ is in the kernel of $\vec{A}_{\vec{G}_{\emptyset}}$. The second and third claims follow from Theorem \ref{zullthmstate}. For the last claim, recall that since $\vec{A}_{\vec{G}_{\emptyset}}$ is a skew-symmetric matrix over $\mathbb{R}$, it is diagonalizable over $\mathbb{C}$.  On the other hand, the first three claims show that $\mathscr{E}(0)$ has a basis consisting of vectors over the rationals.  It follows that the dimension of the eigenspace over the complex numbers is equal to the dimension of the eigenspace over the rational numbers.  This is exactly the algebraic multiplicity.                               
\end{proof}

\subsection{Unoriented Smoothings: Double Covers of $\Sigma_D(S)$.} \label{dblcover_sec} Now consider the case of a partial smoothing $S$ where at least one arrow has the unoriented smoothing. Then $\Sigma_D(S)$ is non-orientable. Thus, the proof of Theorem \ref{zullthmstate} does not apply.  This problem may be solved by taking an orientable double covering of $\Sigma_D(S)$ and interpreting it as a chord diagram.

Let $S=(S_o,S_u,S_{\emptyset})$ be a partial state of $D \in \mathscr{D}(S^1,*)$.  Note that if $S_u$ is non-empty, then the surface $\Sigma_D(S)$ is not orientable. 

We construct an orientable double cover $\Sigma_D^2(S)$ of $\Sigma_D(S)$ as follows. Draw two copies of $\Sigma_D(S)$ side-by-side in $\mathbb{R}^2 \times \mathbb{R}$, where one is shaded blue and the other is shaded red. Consider a twisted band $B$ on the blue $\Sigma_D(S)$ and its corresponding band $B'$ on the red $\Sigma_D(S)$.  We make a horizontal cut on $B$ and the corresponding cut on $B'$.  This divides each into two halves $B_1$,$B_2$ and $B'_1,B'_2$. We sew $B_1$ to $B_2'$ and $B_2$ to $B_1'$ so that the half twist is preserved.  The result is an orientable double cover for $\Sigma_D(S)$. This process is illustrated in Figure \ref{sigma2fig}.

\begin{figure}[h]
\[
\begin{array}{ccc}
 \begin{array}{c}\scalebox{.05}{\psfig{figure=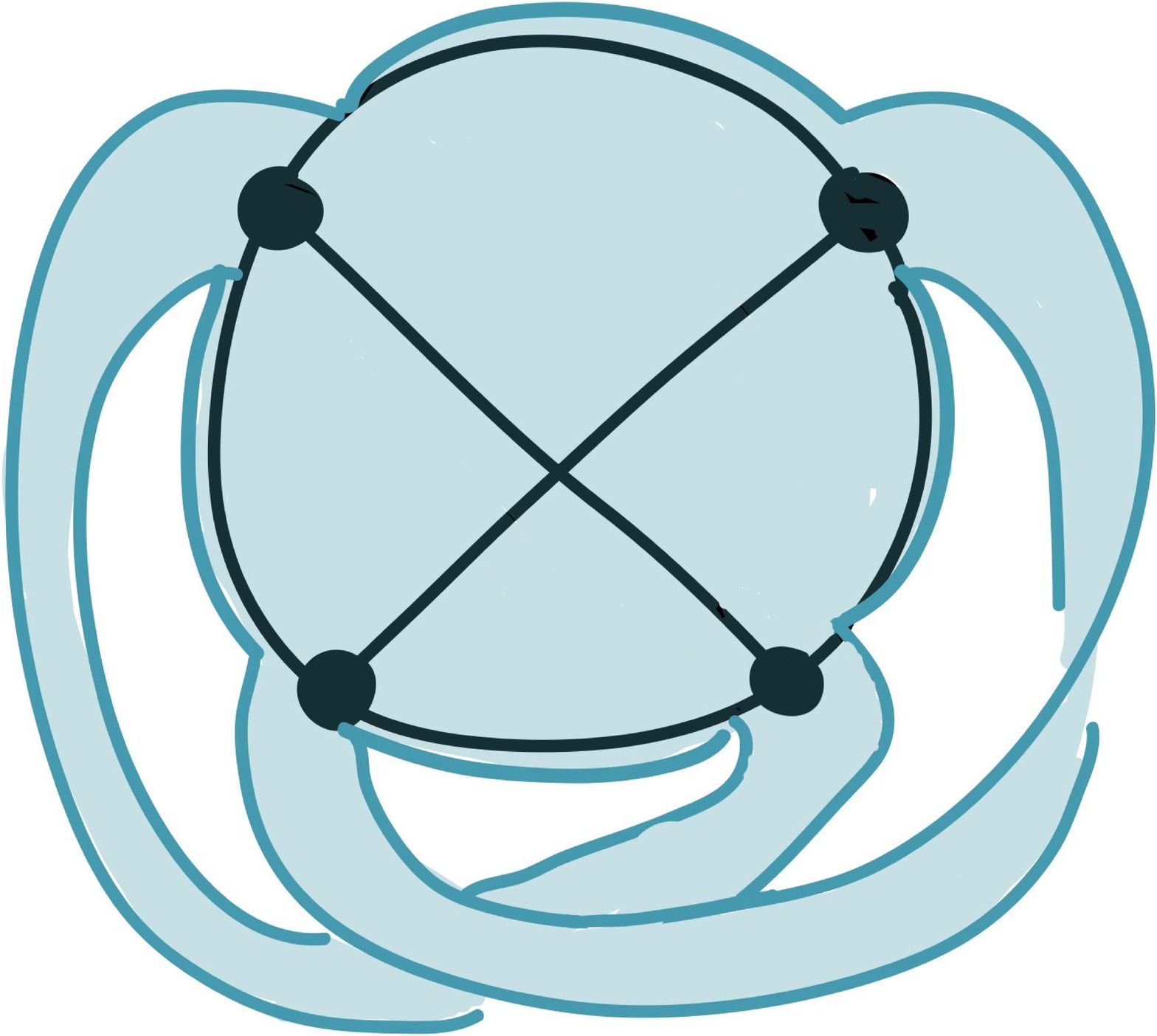}} \end{array} & \begin{array}{c} \scalebox{.05}{\psfig{figure=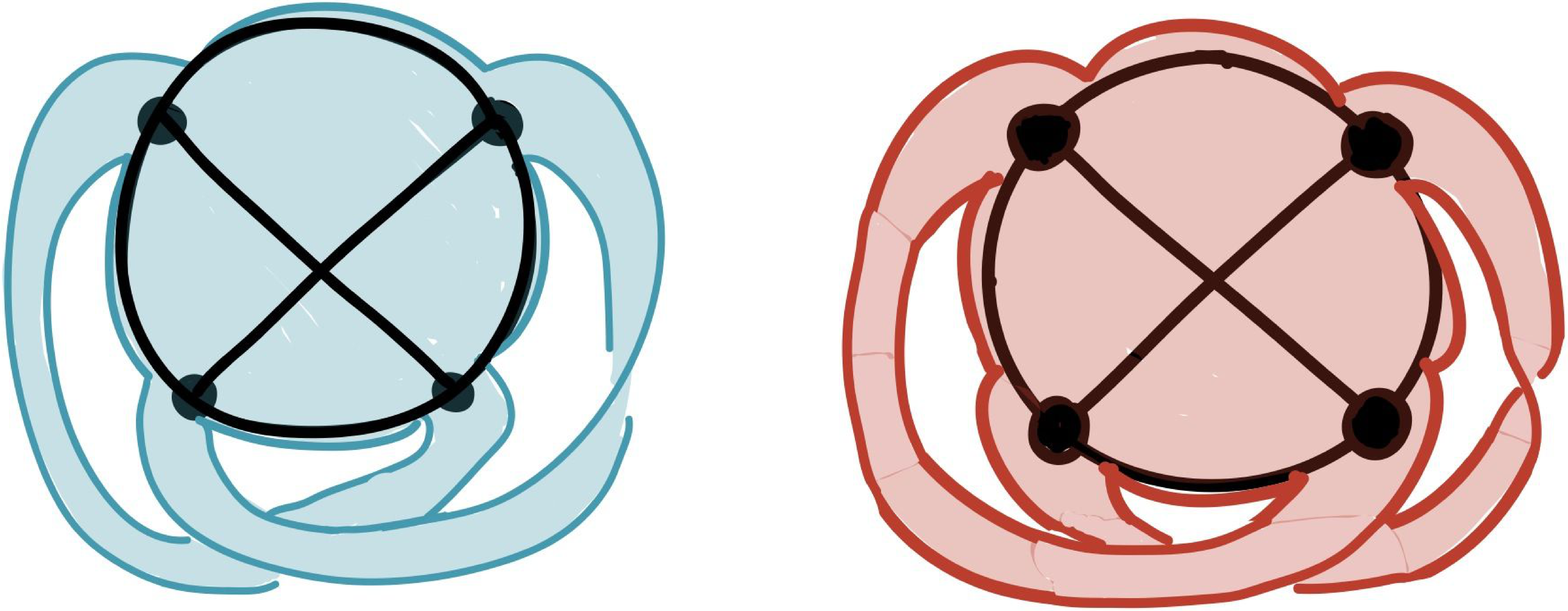}} \end{array} \to & \begin{array}{c} \scalebox{.05}{\psfig{figure=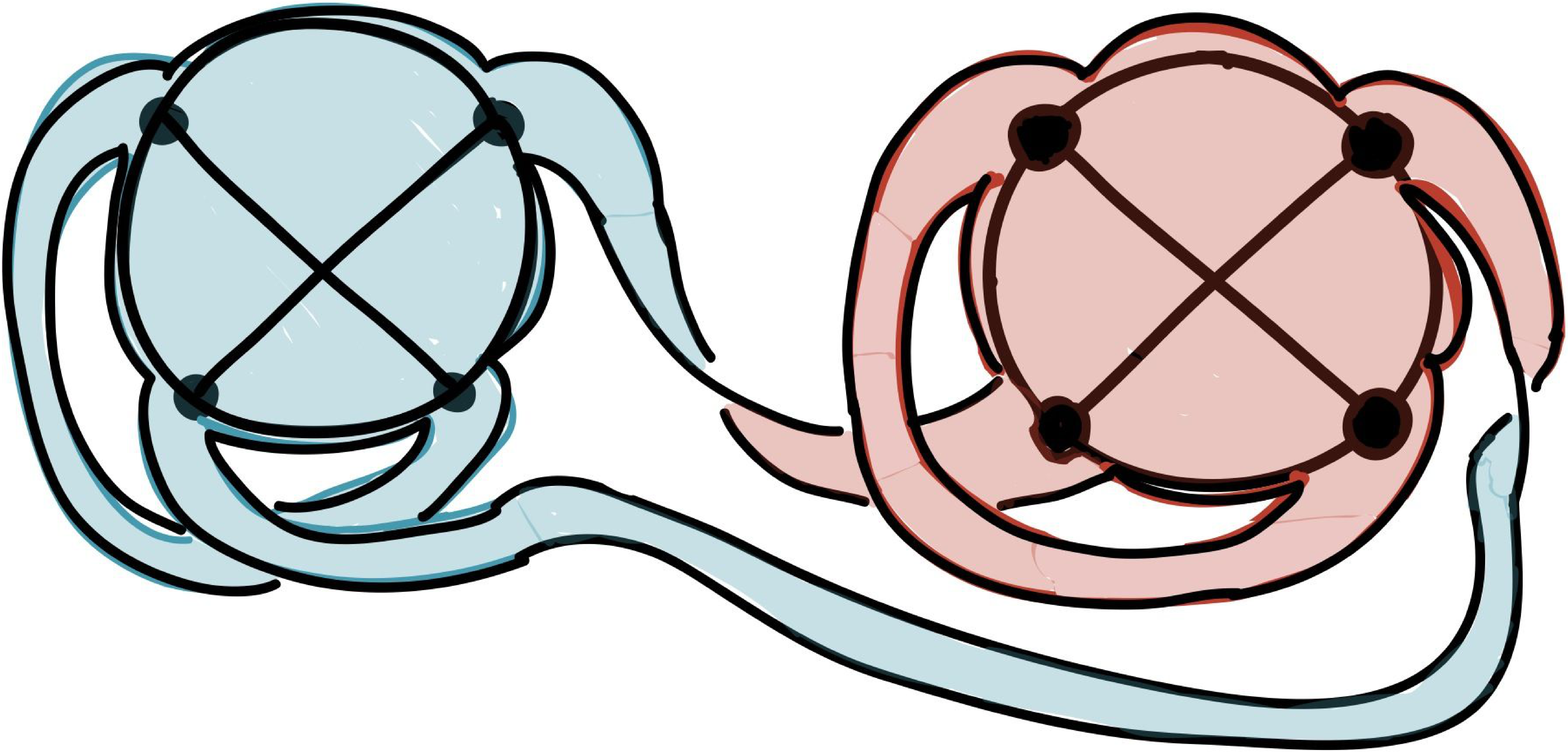}} \end{array}\\
\text{The Non-Orientable Surface } \Sigma_D(S)  & \Sigma_D(S) \text{ and a copy} & \text{Cut and sew}\\ 
\end{array}
\]
\caption{The construction of $\Sigma_D^2(S)$.} \label{sigma2fig}
\end{figure}

Let $S=(S_o,S_u,S_{\emptyset})$ be a partial state of $D$. Let $A$ be the set of letters $a_i$ and $b_i$ and $\overline{A}$ denote the set of letters of the form $\bar{a}_i$, $\bar{b}_i$. Travelling counter-clockwise from $*$, we write $D_{\emptyset}$ as a word in these letters:
\[
\begin{array}{c}\scalebox{.5}{\begin{picture}(0,0)%
\includegraphics{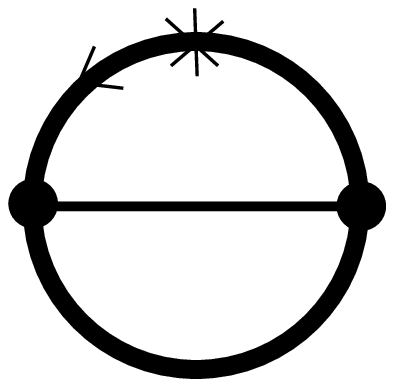}%
\end{picture}%
%
%
\setlength{\unitlength}{3947sp}%
\begingroup\makeatletter\ifx\SetFigFont\undefined%
\gdef\SetFigFont#1#2#3#4#5{%
  \reset@font\fontsize{#1}{#2pt}%
  \fontfamily{#3}\fontseries{#4}\fontshape{#5}%
  \selectfont}%
\fi\endgroup%
\begin{picture}(2155,2242)(349,-1994)
\put(364,-699){\makebox(0,0)[lb]{\smash{{\SetFigFont{20}{24.0}{\rmdefault}{\mddefault}{\updefault}{$a_j$}%
}}}}
\put(2489,-736){\makebox(0,0)[lb]{\smash{{\SetFigFont{20}{24.0}{\rmdefault}{\mddefault}{\updefault}{$b_j$}%
}}}}
\put(451,-49){\makebox(0,0)[lb]{\smash{{\SetFigFont{20}{24.0}{\rmdefault}{\mddefault}{\updefault}{$W_1$}%
}}}}
\put(1189,-1874){\makebox(0,0)[lb]{\smash{{\SetFigFont{20}{24.0}{\rmdefault}{\mddefault}{\updefault}{$W_2$}%
}}}}
\put(2214,-74){\makebox(0,0)[lb]{\smash{{\SetFigFont{20}{24.0}{\rmdefault}{\mddefault}{\updefault}{$W_3$}%
}}}}
\end{picture}%
}\end{array} \to  W_1 a_j W_2 b_j W_3,
\]
where $W_k$ is a possibly empty word in the characters from $A\backslash\{a_j,b_j\}$. For any word $W$ in the letters from $A$, let $\bar{W}$ denote the word $W$ written backwards with each character $x$ from $A$ written as $\bar{x} \in \overline{A}$. We define two words $W^f_j(S)$ and $W^s_j(S)$ as follows.
\begin{eqnarray*}
W^f_j(S) &=& W_1\bar{W}_2\bar{a}_j\bar{W}_1\bar{W}_3 W_2 b_j W_3, \\
W^s_j(S) &=& W_1 a_j W_2 \bar{W}_1 \bar{W}_3 \bar{b}_j \bar{W}_2 W_3.
\end{eqnarray*}
Label $4n-2$ points counter-clockwise on $S^1$ according to the left-to-right order of the letters in the word $W^f_j(S)$ or $W^s_j(S)$.  We define the double cover Gauss diagrams $D_j^f(S)$ or $D_j^s(S)$, respectively, by drawing chords as follows.
\begin{enumerate}
\item For $i \ne j$, if $i\in S_o$, then $a_ib_i$ and $\bar{a}_i \bar{b}_i$ are chords.
\item For $i \ne j$, if $i \in S_u$, then $\bar{a}_ib_i$ and $a_i \bar{b}_i$ are chords.
\item In $D_j^f(S)$, $\bar{a}_j b_j$ is a chord.  In $D_j^s(S)$, $a_j \bar{b}_j$ is a chord.
\end{enumerate} 

\begin{lemma} Let $S=(S_o,S_u,S_{\emptyset})$ be a smoothing with $|S_u| \ge 1$. If $j \in S_u$, then:
\[
\#_{D_j^s(S)}(O)=2 \cdot \#_D(S) =\#_{D_j^f(S)}(O). 
\]
\end{lemma}
\begin{proof} We construct a surface $\Sigma'$ as follows. Draw a copy of $\Sigma_D(S)$ in $\mathbb{R}^2 \times \mathbb{R}$ together with a line $l$ in $\mathbb{R}^2$ which does not intersect $D$. Let $\bar{D}$ denote the reflection of $D$ about $l$. Starting with $*$, label the endpoints of the chords in $D$ as follows. If the chord endpoints and band corners are labeled $a_i$, $b_i$, $a_i',a_i'',b_i',b_i''$ as usual, the corresponding points via reflection through $l$, are labelled $\bar{a}_i,\bar{b}_i,\bar{a}_i',\bar{a}_i'',\bar{b}_i',\bar{b}_i''$ respectively. If the chord $i$ is not in $S_u$, we draw an untwisted band from $a_i$ to $b_i$ and $\bar{a}_i$ to $\bar{b}_i$.  If the chord $i$ is in $S_u$, we draw an \emph{untwisted} band from $a_i$ to $\bar{b}_i$ and an \emph{untwisted} band from $b_i$ to $\bar{a}_i$. For the example in Figure \ref{sigma2fig}, this construction is illustrated on the left hand side of Figure \ref{dubeasy}.

\begin{figure}[h]
\[
\begin{array}{cc}
\scalebox{.09}{\psfig{figure=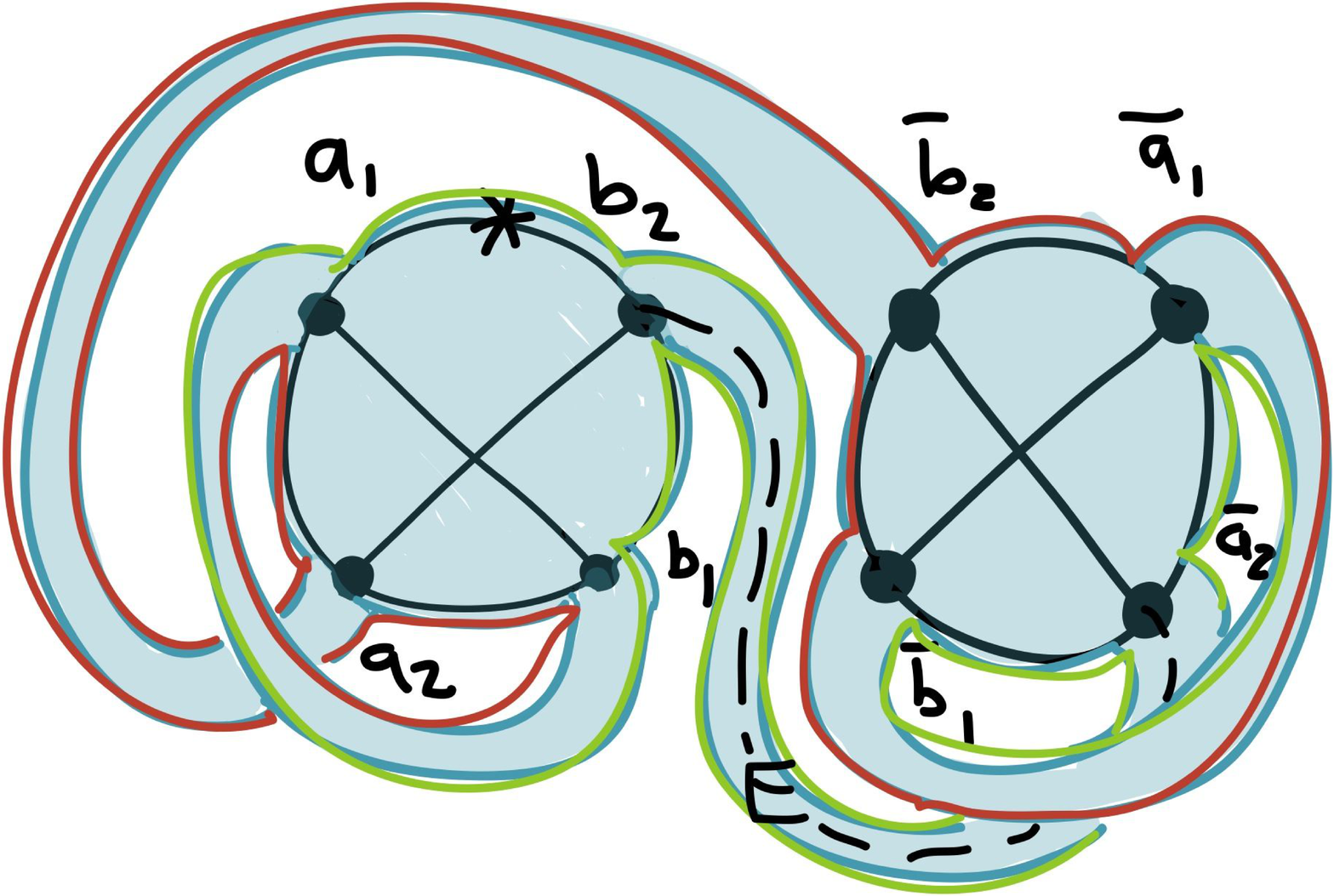}} & \scalebox{.09}{\psfig{figure=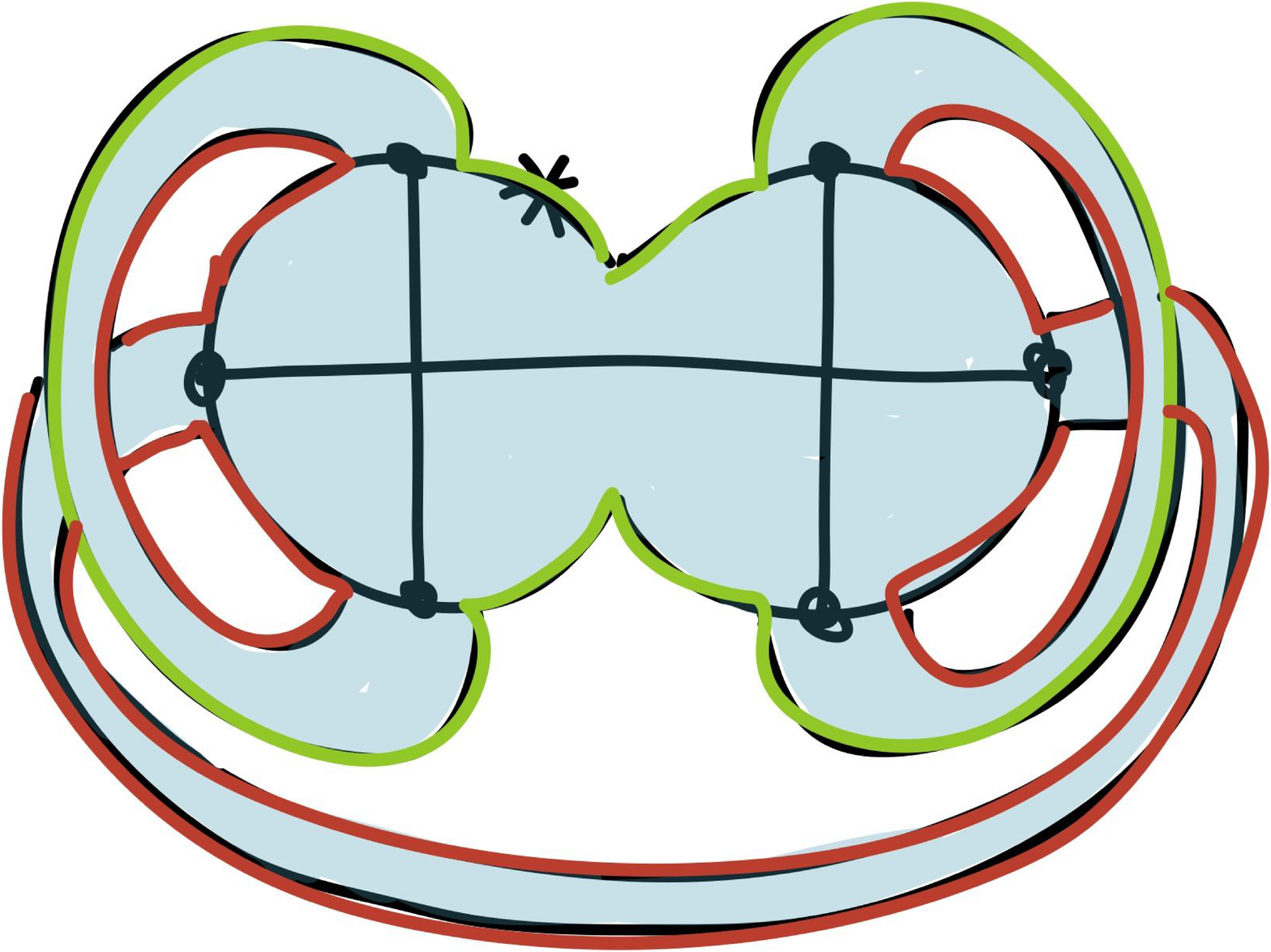}} \\
\end{array}
\]
\caption{Contracting along an edge to obtain a double cover.}\label{dubeasy}
\end{figure}

We must show that $\Sigma'$ is homeomorphic to $\Sigma^2_D(S)$ and that this information is encoded in the smoothing $O$ of both $D_j^f(S)$ and $D_j^s(S)$. First note that for the chord $j$ we have two untwisted bands.  The (f)irst is one from $a_j$ to $\bar{b}_j$ and the (s)econd is the one from $b_j$ to $\bar{a}_j$. First we will choose the second band. It will be denoted by $E$. The band is a rectangle $I \times I$ with one side on the left copy of $S^1$ and the other on the right copy of $S^1$. The other two sides of the rectangle are arcs going from one copy of $S^1$ to the other.  Now imagine taking the following trip.  Move from the basepoint on the first copy of $S^1$ to the point $b_j'$.  Then follow the arc $b_j'\bar{a}_j'$ to the second copy of $S^1$.  Continue counter-clockwise on the $\bar{D}$ copy of $S^1$ until the arc $\bar{a}_j''b_j''$ is reached.  Now continue along this arc to the $D$ copy of $S^1$.  Stop at the basepoint. Since the bands are connected exactly as in $\Sigma_D^2(S)$, we conclude that $\Sigma' \approx \Sigma_D^2(S)$.

Deform the embedding of the surface so that $E$ is contracted to an interval and so that the copies of $S^1$ are glued along this interval (see the right picture in Figure \ref{dubeasy}). This also identifies the endpoint of a chord in $D$ and a chord in $\bar{D}$ to create a new chord (but decreasing the total number of chords by one).  In addition, this gives a distinguished circle $S^1$. Indeed, the circle is the path taken by the trip defined in the previous paragraph. The resulting diagram is $D_j^s(S)$ (up to equivalence of chord diagrams). The smoothing $O$ of each of these diagrams is a surface homeomorphic to $\Sigma^2_D(S)$. Similarly, we may contract along the other untwisted band of $j$ to get $D_j^f(S)$.   
\end{proof}

\section{Computing Characteristic Polynomials} \label{skewspec} The previous sections have reduced the problem of counting loops of partial states  to finding characteristic polynomials of the skew-adjacency matrices of linearly ordered graphs. Fortunately, this partially falls under the purview of an existing mathematical theory.  It is one of the major accomplishments of spectral graph theory \cite{MR2571608,MR1440854}.  Unfortunately, the theory of symmetric adjacency matrices must be redone so that it applies to linearly ordered graphs. Many of the ideas from the symmetric case may nonetheless be applied to our situation. Typically, we will require stronger hypotheses to obtain similar results. The formulas we obtain will turn out to be a little different.

\subsection{Mirror Images} In the symmetric case, the properties of the determinant guarantee that the order of the vertices does not affect the characteristic polynomial. This is not the case for the skew-symmetric adjacency matrix.  However, it is true in the case of the mirror image.

\begin{theorem} \label{mirrorthm} Let $D \in \mathscr{D}(S^1,*)$ and let $\bar{D}$ denote its mirror image. Here, $\bar{D}$ is labelled by the canonical ordering moving CCW from the basepoint.  Let $\vec{G}$ be the linearly ordered intersection graph of $D$ and $\bar{G}$ the linearly ordered intersection graph of $\bar{D}$.  Then:
\[
\vec{P}_{\vec{G}}(x)=\vec{P}_{\bar{G}}(x).
\]
\end{theorem}
\begin{proof} Suppose that $D$ has $n$ chords.  We use the Leibniz formula for the determinant of a matrix. For the label $i$ of $D$, let $\tau(i)$ denote the canonical label of the mirror image of the chord labelled $i$ in $\bar{D}$. Then $\tau \in \mathbb{S}_n$, the symmetric group on $n$ letters.  Let $A$ be the canonical adjacency matrix of $D$ and $\bar{A}$ the canonical adjacency matrix of $\bar{D}$. Then we have that $a_{ij}=\bar{a}_{\tau(j) \tau(i)}$. Let $A'=xI-A$. Then according to the Leibniz formula for $\det(A')$,
\begin{eqnarray*}
\sum_{\sigma \in \mathbb{S}_n} \text{sign}(\sigma) \prod_{i=1}^n a_{i \sigma(i)}' &=&
\sum_{\sigma \tau^{-1} \in (\mathbb{S}_n) \tau^{-1}} \text{sign}(\sigma \tau^{-1}) \prod_{i=1}^n a_{i \sigma\tau^{-1}(i)}' \\
&=& \sum_{\tau^{-1} \sigma \in  \tau^{-1}(\mathbb{S}_n)} \text{sign}(\tau^{-1}\sigma) \prod_{i=1}^n a_{i \tau^{-1} \sigma(i)}' \\
&=& \sum_{\tau^{-1} \sigma \in  \tau^{-1}(\mathbb{S}_n)} \text{sign}(\tau^{-1}\sigma) \prod_{i=1}^n \bar{a}_{\sigma(i)\tau(i)}' \\
&=& \sum_{\tau^{-1} \sigma \in  \tau^{-1}(\mathbb{S}_n)} \text{sign}(\tau^{-1}\sigma) \prod_{j=1}^n \bar{a}_{j \tau \sigma^{-1}(j)}'\\
&=& \sum_{\tau^{-1} \sigma \in  \tau^{-1}(\mathbb{S}_n)} \text{sign}(\tau\sigma^{-1})  \prod_{j=1}^n \bar{a}_{j \tau \sigma^{-1}(j)}' \\
&=& \sum_{\tau \sigma^{-1} \in  \tau(\mathbb{S}_n)^{-1}} \text{sign}(\tau\sigma^{-1}) \prod_{j=1}^n \bar{a}_{j \tau \sigma^{-1}(j)}' \\
&=& \sum_{\gamma \in  \mathbb{S}_n} \text{sign}(\gamma) \prod_{j=1}^n \bar{a}_{j \gamma(j)}' \\
&=& \det(xI-\bar{A})
\end{eqnarray*}
This completes the proof. \end{proof}
\subsection{Adding an Edge} Let $\vec{G}$ be a linearly ordered graph with $n$ vertices. Let $u,v \in V(\vec{G})$ with $l(v)=l(u)+1$ and $u,v$ not adjacent. We denote by $\vec{G}+uv$ the linearly ordered graph obtained from $\vec{G}$ by adding a directed edge from $u$ to $v$. Denote by $\vec{G}-u$ the graph which is obtained from $\vec{G}$ by deleting the vertex $u$ and renumbering in the obvious way.  Finally, for any square matrix $X$, we denote by $\theta_{uv}(X)$ the $(u,v)$ entry of $\text{adj}(X)$, the \emph{adjugate} of $X$. Here, we are using the notation of \cite{MR2571608,MR1440854}. For the results in this section, the reader is invited to compare these with the similar results for the symmetric adjacency matrix case in equations (5.1.4) and (5.1.5) of \cite{MR1440854}.  

\begin{theorem} \label{addedge} With the notations above, the characteristic polynomial of the skew-adjacency matrix of $\vec{G}+uv$ is given by:
\[
\vec{P}_{\vec{G}+uv}(x)=\vec{P}_{\vec{G}-u-v}(x)+\vec{P}_{\vec{G}}(x)+\theta_{uv}(xI-A_{\vec{G}})-\theta_{vu}(xI-A_{\vec{G}}).
\]
\end{theorem}
\begin{proof} The proof is by multi-linear expansion of the determinant. Suppose that the skew-adjacency matrix of $\vec{G}$ is given by:
\[
\vec{A}_{\vec{G}}=\left[ 
\begin{array}{c|c|c|c}
A & -a & -b & B \\ \hline
a^t & 0 & 0 & -c^t \\ \hline
b^t & 0 & 0 & -d^t \\ \hline
C & c & d & D
\end{array}
\right],\,\,\,
xI-\vec{A}_{\vec{G}}=\left[ 
\begin{array}{c|c|c|c}
A' & a & b & B' \\ \hline
-a^t & x & 0 & c^t \\ \hline
-b^t & 0 & x & d^t \\ \hline
C' & -c & -d & D'
\end{array}
\right].
\] 
The multi-linear expansion of $|xI-\vec{A}_{\vec{G}+uv}|$ is then:
\begin{eqnarray*}
\left|
\begin{array}{c|c|c|c}
A' & a & b & B' \\ \hline
-a^t & x & -1 & c^t \\ \hline
-b^t & 1 & x & d^t \\ \hline
C' & -c & -d & D'
\end{array}
\right| &=& \left|
\begin{array}{c|c|c|c}
A' & \vec{0} & b & B' \\ \hline
-a^t & 0 & 0 & c^t \\ \hline
-b^t & 1 & x & d^t \\ \hline
C' & \vec{0} & -d & D'
\end{array}
\right|+\left|
\begin{array}{c|c|c|c}
A' & \vec{0} & \vec{0} & B' \\ \hline
-a^t & 0 & -1 & c^t \\ \hline
-b^t & 1 & 0 & d^t \\ \hline
C' & \vec{0} & \vec{0} & D'
\end{array}
\right|\\
&+& \left|
\begin{array}{c|c|c|c}
A' & a & b & B' \\ \hline
-a^t & x & 0 & c^t \\ \hline
-b^t & 0 & x & d^t \\ \hline
C' & -c & -d & D'
\end{array}
\right|-\left|
\begin{array}{c|c|c|c}
A' & a & \vec{0} & B' \\ \hline
-a^t & x & 1 & c^t \\ \hline
-b^t & 0 & 0 & d^t \\ \hline
C' & -c & \vec{0} & D'
\end{array}
\right|.
\end{eqnarray*}
The first determinant on the right is $\theta_{uv}(xI-\vec{A}_{\vec{G}})$.  The second expression is $\vec{P}_{\vec{G}-u-v}(x)$.  The third expression is $\vec{P}_{\vec{G}}(x)$. The last expression is $\theta_{vu}(xI-\vec{A}_{\vec{G}})$.  This completes the proof of the theorem.
\end{proof}
\begin{corollary} \label{edgdelcorr} Let $\vec{G}$ be a linearly ordered graph with vertices $u$, $v$, such that  $l(v)=l(u)+1$, $u \not\sim v$, and a vertex in $\vec{G}$ is adjacent to $u$ if and only if it is adjacent to $v$.  Then:
\[
\vec{P}_{\vec{G}+uv}(x)=\vec{P}_{\vec{G}-u-v}(x)+\vec{P}_{\vec{G}}(x).
\]
\end{corollary}
\begin{proof} Consider the multi-linear expansion given in Theorem \ref{addedge}.  By hypothesis, we have $a=b$ and $c=d$.  Hence, we have that $\theta_{uv}(xI-\vec{A}_{\vec{G}})=\theta_{vu}(xI-\vec{A}_{\vec{G}})$. 
\end{proof}   
\subsection{Joins} Let $\vec{G}_1$ and $\vec{G}_2$ be linearly ordered graphs with $n_1$ and $n_2$ vertices respectively.  We may form the linearly ordered graph $\vec{G}_1 \sqcup \vec{G}_2$ as follows. The underlying graph is the disjoint union of $G_1$ and $G_2$.  Vertices corresponding to those in $G_1$ are labelled as in $\vec{G}_1$. A vertex $v$ in $G_2$ is labelled $l(v)+n_1$ in $\vec{G}_1 \sqcup \vec{G}_2$. The following result is exactly the same as the symmetric adjacency matrix case \cite{MR2571608}.

\begin{theorem} The characteristic polynomial of $\vec{G}_1 \sqcup \vec{G}_2$ is given by:
\[
\vec{P}_{\vec{G}_1 \sqcup \vec{G}_2}(x)=\vec{P}_{\vec{G}_1}(x)\cdot \vec{P}_{\vec{G}_2}(x).
\]
\end{theorem}

Suppose that we are given two linearly ordered graphs $\vec{G}_1$ and $\vec{G}_2$ with $n_1$ and $n_2$ vertices respectively. We may form the join, denoted, $\vec{G}_1 \vec{\nabla} \vec{G}_2$ by taking the disjoint union of $\vec{G}_1$ and $\vec{G}_2$, relabelling every vertex $v$ of $\vec{G}_2$ by $l(v)+n_1$ and connecting every vertex of $\vec{G}_1$ with every vertex of $\vec{G}_2$. The new edges are directed from $u \in V(\vec{G}_1)$ towards $v \in V(\vec{G}_2)$. Then the skew-adjacency matrix is given by:
\[
\left[ \begin{array}{cc}
\vec{A}_1 & J \\
-J^t  & \vec{A}_2
\end{array}
\right],
\] 
where $J$ denotes the matrix of appropriate dimensions having all ones.  An example of the join construction is given in Figure \ref{joinfig}.

\begin{figure}
\[
\begin{array}{ccc}
\scalebox{.5}{\psfig{figure=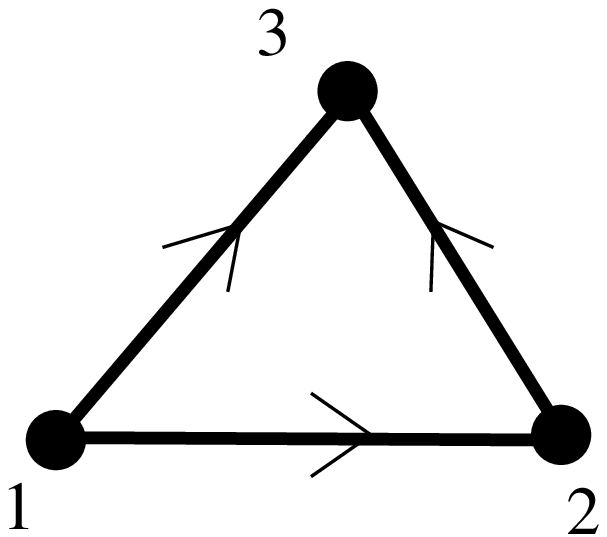}} & \scalebox{.5}{\psfig{figure=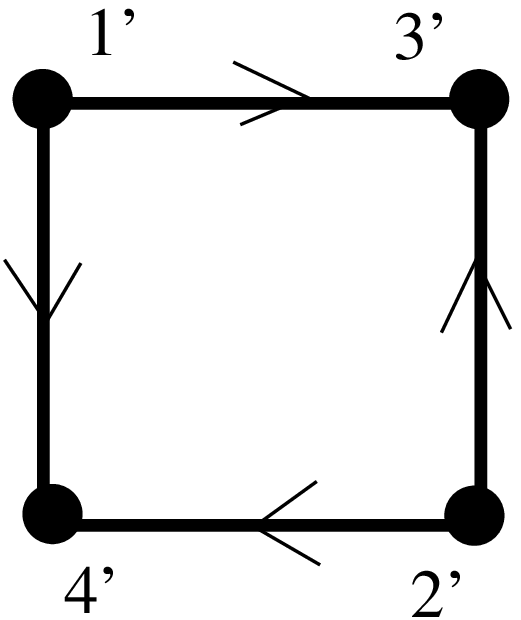}} & \scalebox{.35}{\psfig{figure=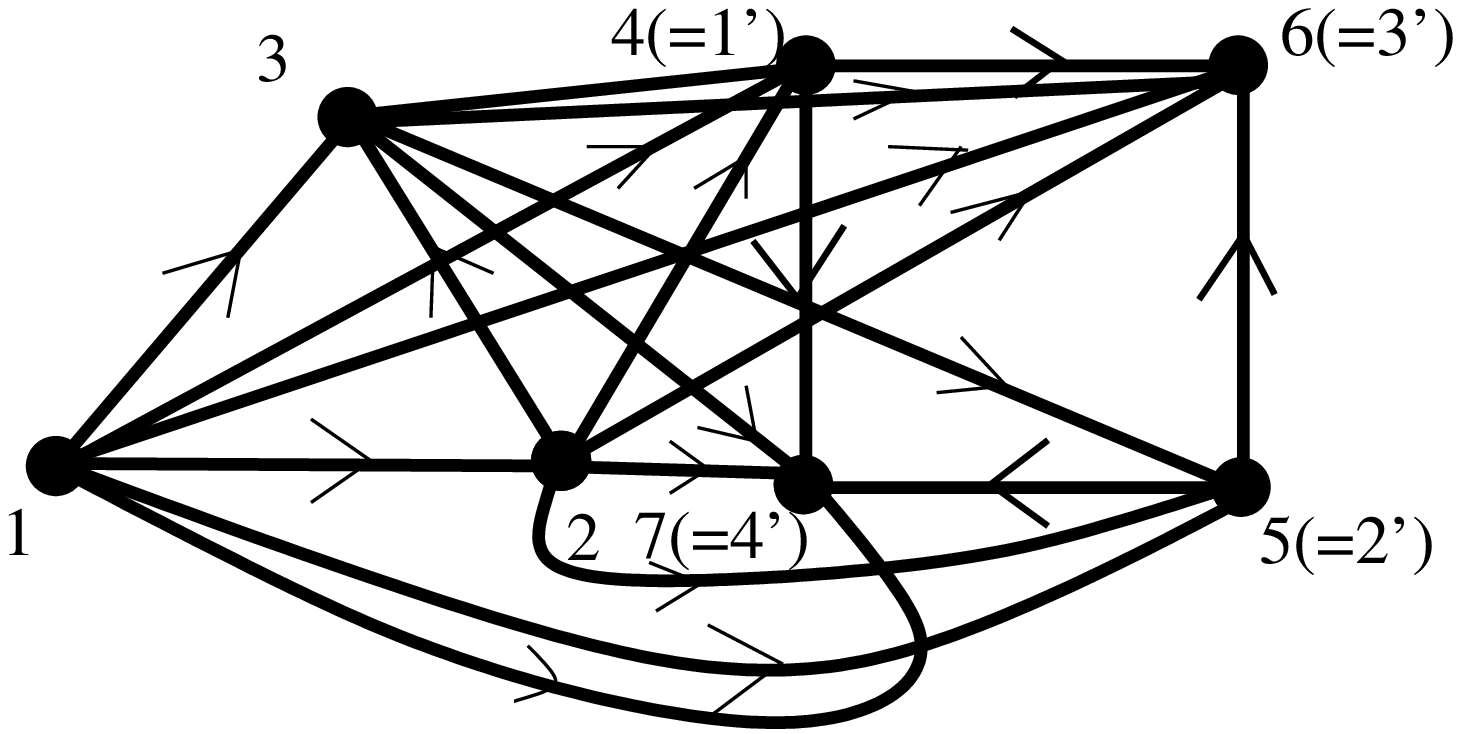}} \\
\vec{G} & \vec{H} & \vec{G} \vec{\nabla} \vec{H}
\end{array}
\]
\caption{Join of $\vec{G}$ and $\vec{H}$.} \label{joinfig}
\end{figure}

The reader should compare this notion of join with the notion of join in \cite{MR2571608} (Theorem 2.1.5, Corollary 2.1.6, and Proposition 2.1.7).
 
\begin{theorem} \label{jointhm} The characteristic polynomial of the skew adjacency matrix of $\vec{G} \vec{\nabla} \vec{H}$ is given by:
\[
\vec{P}_{\vec{G} \vec{\nabla} \vec{H}}(x)=\vec{P}_{\vec{G}}(x)\vec{P}_{\vec{H}}(x)+(\vec{P}_{K_1 \vec{\nabla} \vec{G}}(x)-x \vec{P}_{\vec{G}}(x))(\vec{P}_{K_1 \vec{\nabla} \vec{H}}(x)-x \vec{P}_{\vec{H}}(x)).
\]
\end{theorem}
\begin{proof} Suppose that $|\vec{G}|=n\ge 1$ and $|\vec{H}|=m\ge 1$.  Let $N=n+m$.  The proof is by induction on $N$.  If $N=2$, then $n=m=1$.  In this case, the left and right hand sides of the equation are both $x^2+1$.

Now suppose that the result is true up to $N-1$, where $N-1 \ge 1$ and that $\vec{G}\vec{\nabla}\vec{H}$ has $N=n+m$ vertices.  We will first show that the derivative of the left-hand side is equal to the derivative of the right-hand side.  Then we will show that both sides have the same value at $0$. Note that the derivative may be expressed as:
\[
\vec{P}_{\vec{G}}'(x)=\sum_{j=1}^n \vec{P}_{\vec{G}-j}(x),
\] 
where $\vec{G}-j$ represents the linearly ordered graph $\vec{G}$ with the vertex $j$ deleted and all the remaining vertices are relabelled appropriately (subtract one from the label of all vertices with label greater than $l(j)$). The proof of this fact is identical to the proof of Theorem 2.3.1 in \cite{MR2571608} (also, see this book for further references).

It follows that the derivative of the left hand side is given by:
\[
\vec{P}_{\vec{G} \vec{\nabla} \vec{H}}'(x)=\sum_{j \in V(\vec{G} \vec{\nabla} \vec{H})} \vec{P}_{(\vec{G} \vec{\nabla} \vec{H})-j}(x)=\sum_{j \in V(\vec{G})} \vec{P}_{(\vec{G}-j) \vec{\nabla} \vec{H}}(x)+\sum_{j \in V(\vec{H})} \vec{P}_{\vec{G} \vec{\nabla} (\vec{H}-j)}(x).
\]
We can apply the induction hypothesis to each of these summands.  This gives:
\begin{eqnarray*}
\vec{P}_{(\vec{G}-j) \vec{\nabla} \vec{H}}(x) &=& \vec{P}_{\vec{G}-j}(x)\vec{P}_{\vec{H}}(x)+(\vec{P}_{K_1 \vec{\nabla} (\vec{G}-j)}(x)-x \vec{P}_{\vec{G}-j}(x))(\vec{P}_{K_1 \vec{\nabla} \vec{H}}(x)-x \vec{P}_{\vec{H}}(x)),\\
\vec{P}_{\vec{G} \vec{\nabla} (\vec{H}-j)}(x) &=& \vec{P}_{\vec{G}}(x)\vec{P}_{\vec{H}-j}(x)+(\vec{P}_{K_1 \vec{\nabla} \vec{G}}(x)-x \vec{P}_{\vec{G}}(x))(\vec{P}_{K_1 \vec{\nabla}(\vec{H}-j)}(x)-x \vec{P}_{\vec{H}-j}(x)). \\
\end{eqnarray*}
The derivative of the right hand side of the equation is as follows:
\begin{eqnarray*}
\vec{P}_{\vec{H}}(x)\sum_{j \in V(\vec{G})} \vec{P}_{\vec{G}-j}(x)+\vec{P}_{\vec{G}}(x)\sum_{j \in V(\vec{H})} \vec{P}_{\vec{H}-j}(x) &+&(\vec{P}_{K_1 \vec{\nabla} \vec{G}}(x)-x \vec{P}_{\vec{G}}(x))'(\vec{P}_{K_1 \vec{\nabla} \vec{H}}(x)-x \vec{P}_{\vec{H}}(x)) \\
&+& (\vec{P}_{K_1 \vec{\nabla} (\vec{G})}(x)-x \vec{P}_{\vec{G}}(x))(\vec{P}_{K_1 \vec{\nabla} \vec{H}}(x)-x \vec{P}_{\vec{H}}(x))'. \\
\end{eqnarray*}
Now we take the derivative of the terms in parentheses. For $\vec{C}=\vec{G}$ or $\vec{H}$, we obtain:
\begin{eqnarray*}
(\vec{P}_{K_1 \vec{\nabla} \vec{C}}(x)-x \vec{P}_{\vec{C}}(x))'&=& \vec{P}_{K_1 \vec{\nabla} \vec{C}}'(x)-x \cdot \vec{P}_{\vec{C}}'(x)-\vec{P}_{\vec{C}}(x)\\
&=& \sum_{j \in V(K_1\vec{\nabla}\vec{C})} \vec{P}_{(K_1 \vec{\nabla} \vec{C})-j}(x)-x \cdot \sum_{j \in V(\vec{C})} \vec{P}_{\vec{C}-j}(x)-\vec{P}_{\vec{C}}(x)\\
&=& \sum_{j \in V(\vec{C})} \vec{P}_{K_1 \vec{\nabla} (\vec{C}-j)}(x)-x \cdot \sum_{j \in V(\vec{C})} \vec{P}_{\vec{C}-j}(x).
\end{eqnarray*}
Comparing these expansions of the derivatives of the left and right hand sides reveals that their derivatives are identical. To finish the proof, it needs to be shown that:
\[
\vec{P}_{\vec{G} \vec{\nabla} \vec{H}}(0)=\vec{P}_{\vec{G}}(0)\vec{P}_{\vec{H}}(0)+\vec{P}_{K_1 \vec{\nabla} \vec{G}}(0)\vec{P}_{K_1 \vec{\nabla} \vec{H}}(0).
\]
Ultimately, we will use the Leibniz formula to justify this claim.  Let $A$ denote the skew-adjacency matrix of $\vec{G}$ and $B$ denote the skew-adjacency matrix of $\vec{H}$. Note that the determinant of the skew-adjacency matrix of $K_1 \vec{\nabla} \vec{G}$, where the vertex labelled $1$ corresponds to the vertex added to $\vec{G}$, is $\vec{j}^t_n \text{adj}(A) \vec{j}_n$.  For $K_1 \vec{\nabla} \vec{H}$, the determinant is $\vec{j}^t_m\text{adj}(B) \vec{j}_m$, where $\vec{j}_m$ is the $m \times 1$ matrix of all ones.

Let $M$ denote the skew-adjacency matrix of $\vec{G}\vec{\nabla}\vec{H}$. Recall the Leibniz formula for the determinant of a matrix:
\[
\det(M)=\left|\begin{array}{cc}
A & J \\
-J^t  & B
\end{array}
\right|= \sum_{\sigma \in \mathbb{S}_{n+m}} \text{sign}(\sigma) \prod_{i=1}^{n+m} m_{i \sigma(i)}.
\]
Let $\sigma \in \mathbb{S}_{n+m}$ and let $k$ be the number of $i$ between $1$ and $n$ such that $\sigma(i)>n$. Let $X=\{i_1,\ldots,i_k\}$ denote the set of such elements.  
\newline
\newline
\underline{Claim:} If $k \ge 2$,then there is a $\tau \in \mathbb{S}_{n+m}$ such that:
\[
\text{sign}(\sigma) \prod_{i=1}^{n+m} m_{i \sigma(i)}+\text{sign}(\tau) \prod_{i=1}^{n+m} m_{i \tau(i)}=0.
\] 
There must also be exactly $k$ elements $j$ from $n+1$ to $m$ such that $\sigma(j)<n+1$. Denote the set of these by $Y=\{j_1,\ldots,j_k\}$.  Note that for $i\in X$ and $j\in Y$, we have $m_{i\sigma(i)}=1$ and $m_{j \sigma(j)}=-1$. Thus, if we take any permutation of the elements of $\sigma(X)$ or a permutation of the elements of $\sigma(Y)$ and compose it with $\sigma$ to obtain a new permutation $\gamma \in \mathbb{S}_{n+m}$, then the term below will be unaffected:
\[
\prod_{i=1}^{n+m} m_{i \gamma(i)}.
\]
Since $k \ge 2$, the result creates as many even permutations as odd permutations. This completes the proof of the claim. $\hfill\boxdot$

Now, if $k=0$, then $\sigma=\tau \gamma$ where $\tau$, $\gamma$ are disjoint, $\tau \in \mathbb{S}_n$ and $\gamma$ is a permutation of the elements $\{n+1,\ldots,m\}$. It follows that:
\[
\sum_{\sigma \in \mathbb{S}_{n+m}, k=0} \text{sign}(\sigma) \prod_{i=1}^{m+n} m_{i\sigma(i)}=\det(A)\det(B).
\]

Now suppose that $k=1$.  In this case, $\prod_{i=1}^{m+n} m_{i\sigma(i)}$ consists of $n-1$ elements of $A$, $m-1$ elements of $B$, a $1$, and a $-1$. The $1$ specifies a row coordinate $i$, $1 \le i \le n$, and a column coordinate $j$, $n+1\le j \le m$.  The $-1$ specifies a row coordinate $i'$, $n+1 \le i' \le m$ and a column coordinate $j'$, $1 \le j' \le n$. Then these together specify the $(i,j')$ cofactor of $A$ and the $(i',j)$ cofactor of $B$.

We will set-up a one-to-one correspondence with permutations of the form $\sigma_1 \sigma_2$ such that $\sigma_1 \in \mathbb{S}(1,\ldots,n)$, $\sigma_2 \in \mathbb{S}(n+1,\ldots,m)$, and $\text{sign}(\sigma)\text{sign}(\sigma_1\sigma_2)=-1$. Indeed, using the notation as above, define $\sigma_1$ to be $\sigma|X$ and $\sigma_1(i)=j'$.  Define $\sigma_2$ to be $\sigma|Y$ and $\sigma(i')=j$. Now, take the disjoint cycle decompositions of $\sigma_1$ and $\sigma_2$.  Let $\tau_1$ be the cycle of $\sigma_1$ which contains $i$ and $\tau_2$ the cycle of $\sigma_2$ which contains $i'$. Note that $\tau_1 (i j) \tau_2$ is a permutation which sends $i \to j$ and $i' \to j'$.  It follows that $\sigma=\sigma_1 (i j) \sigma_2$ and that $\text{sign}(\sigma)\text{sign}(\sigma_1\sigma_2)=-1$.   

Hence, this defines a one-to-one correspondence between the summands of the elements of $\det(M)$ with $k=1$ and the product:
\[
(\vec{j}^t_n \text{adj}(A)\vec{j}_n)(\vec{j}^t_m \text{adj}(B)\vec{j}_m).
\]
Note: The extra $-1$ adjusts for the fact the permutation in $M$ contains an additional transposition $(ij)$.

It follows the desired formula evaluated at $0$ is always true. Thus our argument by differentiation shows that the formula holds for all $N$.  This completes the proof by mathematical induction.
\end{proof}
\subsection{Coalescence} Let $G$ and $H$ be disjoint undirected graphs having vertices $u$ and $v$ respectively. The coalescence of $G$ and $H$ at $u$ and $v$ is the graph obtained by identifying the vertices $u$ and $v$.  The coalescence is denoted $G \cdot H$. For linearly ordered graphs $\vec{G}$ and $\vec{H}$ with vertices $u$ and $v$ respectively, the coalescence $\vec{G} \cdot \vec{H}$ is defined to be $G\cdot H$ with all the vertices of $G$ labelled as in $\vec{G}$ and each vertex $w$ of $\vec{H}-v$ labelled as $l(w)+n$.  The vertex $u(=v)$ is labelled $u$. Since $u$ is less than every vertex in $\vec{H}-v$ to which it is adjacent, all of the edges from $u$ to $x \in V(\vec{H}-v)$ are directed $u \to x$. This is illustrated in Figure \ref{coalfig}.

\begin{figure}
\[
\begin{array}{ccc}
\scalebox{.5}{\psfig{figure=k3directed.eps}} & \scalebox{.5}{\psfig{figure=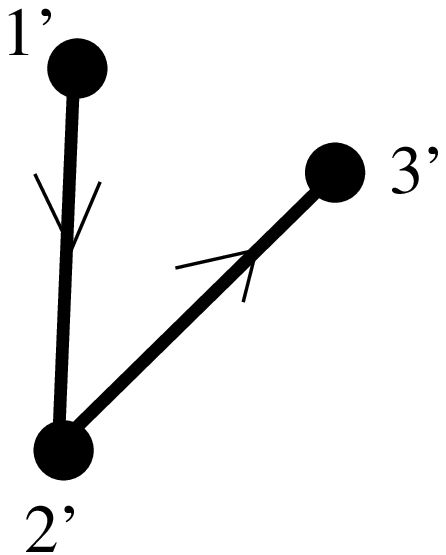}} & \scalebox{.5}{\psfig{figure=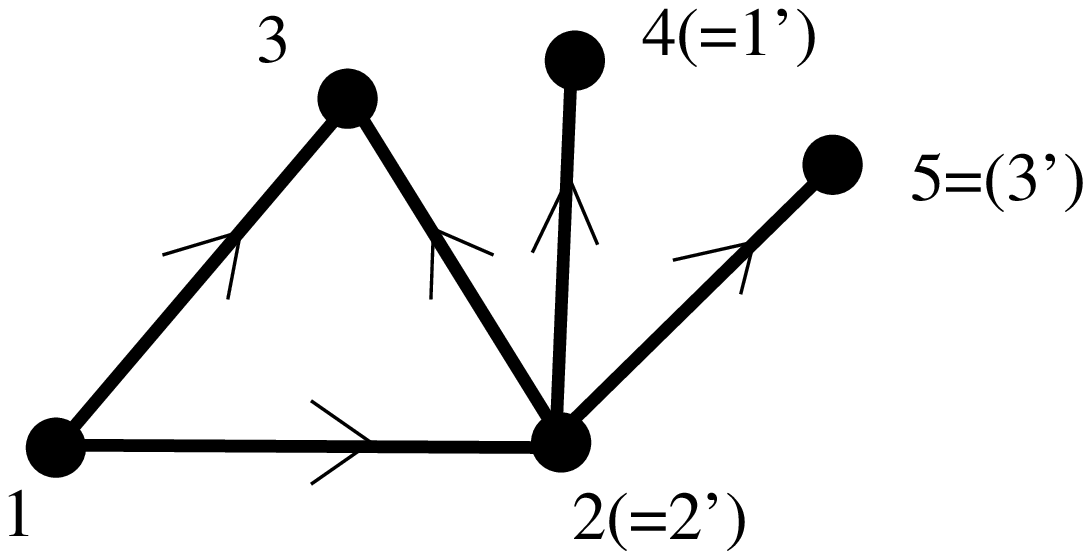}} \\
\vec{G} & \vec{H} & \vec{G} \cdot \vec{H}
\end{array}
\]
\caption{Coalescing over the vertex $2$ of $\vec{G}$ and $2'$ of $\vec{H}$.} \label{coalfig}
\end{figure}

Let $\vec{G}$ be any linearly ordered graph and $u$ any vertex of $\vec{G}$.  We define the \emph{promotion} of $u$ in $\vec{G}$ to be the linearly ordered graph obtained from $\vec{G}$ by deleting the vertex $u$, adding a vertex $v_0$ (with $l(v_0)=1$) having the same adjacent vertices as $u$, and relabelling any other vertex $w$ with $l(w)<l(u)$ as $l(w)+1$. We will denote the promotion of $u$ in $\vec{G}$ as $\vec{G}\leftrightarrow u$.

The next result shows how the skew-spectrum of the coalescence is related to the skew-spectra of simpler linearly ordered graphs. The proof is very similar to the proof of Theorem 2.2.3 in \cite{MR2571608}.
 
\begin{theorem} \label{coalthm} The characteristic polynomial of the coalescence of two linearly ordered graphs $\vec{G}$ and $\vec{H}$ at $u$ and $v$ is given by:
\[
\vec{P}_{\vec{G}\cdot\vec{H}}(x)=\vec{P}_{\vec{G}}(x)\vec{P}_{\vec{H}-v}(x)+\vec{P}_{\vec{G}-u}(x)\vec{P}_{\vec{H} \leftrightarrow v}(x)-x \vec{P}_{\vec{G}-u}(x)\vec{P}_{\vec{H}-v}(x).
\]
\end{theorem}  
\begin{proof} Suppose that the skew-adjacency matrices of $\vec{G}$ and $\vec{H}$ are given as below:
\[
\left[\begin{array}{ccc} A & \alpha & B \\ -\alpha^t & 0 & \beta^t \\ C & -\beta & D \end{array}\right],\,\,\,\left[\begin{array}{ccc} E & \gamma & F \\ -\gamma^t & 0 & \delta^t \\ C & -\delta & D \end{array}\right].
\]
Here the ``Greek columns'' in the matrices represent the skew-adjacencies of $u$ and $v$ respectively. Let $M$ represent the skew-adjacency matrix of the linearly ordered graph $\vec{G}\cdot\vec{H}$.  Then $xI-M$ is given by:
\[
\left[\begin{array}{ccccc} 
A' & -\alpha & B' & \vec{0} & \vec{0} \\ 
\alpha^t & x & -\beta^t & -\gamma^t & -\delta^t \\
C' & \beta & D' & \vec{0} & \vec{0} \\
\vec{0} & \gamma & \vec{0} & E' & F' \\
\vec{0} & \delta & \vec{0} & G' & H'
\end{array}\right],
\]
where $A',B',C',D',E',F',G',H',I'$ are the matrices obtained from $A,B,C,D,E,F,G,H,I$ upon subtracting $M$ from $xI$.  Since the determinant is multi-linear, it follows that $\det(xI-M)$ may be expressed as:
\[
\left|\begin{array}{ccccc} 
A' & \vec{0} & B' & \vec{0} & \vec{0} \\ 
\alpha^t & x & -\beta^t & -\gamma^t & -\delta^t \\
C' & \vec{0} & D' & \vec{0} & \vec{0} \\
\vec{0} & \gamma & \vec{0} & E' & F' \\
\vec{0} & \delta & \vec{0} & G' & H'
\end{array}\right|+
\left|\begin{array}{ccccc} 
A' & -\alpha & B' & \vec{0} & \vec{0} \\ 
\alpha^t & x & -\beta^t & -\gamma^t & -\delta^t \\
C' & \beta & D' & \vec{0} & \vec{0} \\
\vec{0} & \vec{0} & \vec{0} & E' & F' \\
\vec{0} & \vec{0} & \vec{0} & G' & H'
\end{array}\right|-
\left|\begin{array}{ccccc} 
A' & \vec{0} & B' & \vec{0} & \vec{0} \\ 
\alpha^t & x & -\beta^t & -\gamma^t & -\delta^t \\
C' & \vec{0} & D' & \vec{0} & \vec{0} \\
\vec{0} & \vec{0} & \vec{0} & E' & F' \\
\vec{0} & \vec{0} & \vec{0} & G' & H'
\end{array}\right|.
\]  
The second determinant is $\vec{P}_{\vec{G}}(x)\vec{P}_{\vec{H}-v}(x)$.  The third determinant is $x \vec{P}_{\vec{G}-u}(x)\vec{P}_{\vec{H}-v}(x)$.  The first determinant may be rearranged as follows to obtain $\vec{P}_{\vec{G}-u}(x)\vec{P}_{\vec{H} \leftrightarrow v}(x)$:
\[
\left|\begin{array}{ccccc} 
A' & \vec{0} & B' & \vec{0} & \vec{0} \\ 
\alpha^t & x & -\beta^t & -\gamma^t & -\delta^t \\
C' & \vec{0} & D' & \vec{0} & \vec{0} \\
\vec{0} & \gamma & \vec{0} & E' & F' \\
\vec{0} & \delta & \vec{0} & G' & H'
\end{array}\right|=
\left|\begin{array}{ccccc} 
A' & B' & \vec{0} & \vec{0} & \vec{0} \\ 
C' & D' & \vec{0} & \vec{0} & \vec{0} \\
\alpha^t & -\beta^t & x & -\gamma^t & -\delta^t \\
\vec{0} & \vec{0} & \gamma & E' & F' \\
\vec{0} & \vec{0} & \delta & G' & H'
\end{array}\right|.
\]
\end{proof}
\subsection{Building Block: Linear Ordered Paths.} Let $D_n$ denote the based chord diagram whose labelled chord endpoints from the basepoint are given by the code:
\[
12132435465\cdots(n-1)(n-2)n(n-1)n.
\] 
Then the intersection graph of $D_n$ is the ordered path $\vec{P}_n$.  The skew-adjacency matrix of $D_n$ is given by:
\[
\vec{A}_n=\left[ \begin{array}{cccccc} 0 & 1 & 0 & \cdots & 0 & 0 \\ -1 & 0 & 1 & \cdots & 0 & 0 \\ 
0 & -1 & 0 & \cdots & 0 & 0 \\
\vdots & \vdots & \vdots & \ddots & \vdots & \vdots \\ 
0 & 0 & 0 & \cdots & 0 & 1 \\
0 & 0 & 0 & \cdots & -1 & 0 \end{array}\right].
\]
\begin{theorem} Let $n \in \mathbb{N}$.  Then the ordered path $\vec{P}_n$ satisfies:
\begin{enumerate}
\item The recurrence relation $\vec{P}_{\vec{P}_n}(x)=x \cdot \vec{P}_{\vec{P}_{n-1}}(x)+\vec{P}_{\vec{P}_{n-2}}(x)$ with initial conditions $\vec{P}_{\vec{P}_1}(x)=x$ and $\vec{P}_{\vec{P}_2}(x)=x^2+1$.
\item The solution to this recurrence relation is:
\[
\vec{P}_{\vec{P}_n}(x)= \frac{1}{2^{n+1}}\left[\left(1+\frac{x}{\sqrt{x^2+4}}\right)\left(x+\sqrt{x^2+4}\right)^n+\left(1-\frac{x}{\sqrt{x^2+4}}\right)\left(x-\sqrt{x^2+4} \right)^n\right].
\]
\item The solution may be expressed as a sum:
\[
\vec{P}_{\vec{P}_n}(x)=\frac{1}{2^n}\sum_{j=0}^{\lfloor n/2 \rfloor} \left(\begin{array}{c} n+1 \\ 2j+1 \end{array}\right) x^{n-2j} (x^2+4)^j.
\]
\end{enumerate}
\end{theorem}
\begin{proof}For the first, we simply compute $\det(xI-\vec{A}_n)$.
\begin{eqnarray*}
\left|\begin{array}{ccccccc} x & -1 & 0 & 0 & \cdots & 0 & 0 \\ 1 & x & -1 & 0 & \cdots & 0 & 0\\ 0 & 1 & x & -1 & \cdots & 0 & 0 \\ \vdots & \ddots & \ddots & \ddots & \cdots & \vdots & \vdots \\0 & 0 & 0 & 0 & \cdots & x & -1 \\ 0 & 0 & 0 & 0 & \cdots & 1 & x \\ \end{array}\right| &=& x \cdot \vec{P}_{P_{n-1}}(x)+\left|\begin{array}{ccccccc} 1 & -1 & 0 & 0 & \cdots & 0 & 0 \\ 0 & x & -1 & 0 & \cdots & 0 & 0\\ 0 & 1 & x & -1 & \cdots & 0 & 0 \\ \vdots & \ddots & \ddots & \ddots & \cdots & \vdots & \vdots \\0 & 0 & 0 & 0 & \cdots & x & -1 \\ 0 & 0 & 0 & 0 & \cdots & 1 & x \\ \end{array}\right|\\
&=& x \cdot \vec{P}_{\vec{P}_{n-1}}(x)+\vec{P}_{\vec{P}_{n-2}}(x)+\left|\begin{array}{ccccccc} 0 & -1 & 0 & 0 & \cdots & 0 & 0 \\ 0 & x & -1 & 0 & \cdots & 0 & 0\\ 0 & 1 & x & -1 & \cdots & 0 & 0 \\ \vdots & \ddots & \ddots & \ddots & \cdots & \vdots & \vdots \\0 & 0 & 0 & 0 & \cdots & x & -1 \\ 0 & 0 & 0 & 0 & \cdots & 1 & x \\ \end{array}\right| \\
&=& x \cdot \vec{P}_{\vec{P}_{n-1}}(x)+\vec{P}_{\vec{P}_{n-2}}(x).\\
\end{eqnarray*}
The second claim follows from solving the recurrence relation. The third claim follows from an application of the Binomial Theorem and Pascal's Triangle formula.
\end{proof}
\subsection{Building Block: Linearly Ordered Complete Graphs} Let $D_n$ denote the based chord diagram whose labelled chord endpoints from the basepoint are given by the code:
\[
1 2 3 \cdots n 1 2 3 \cdots n.
\]
The intersection graph of $D_n$ is the complete graph on $n$ vertices. We will denote this graph with its canonical ordering as $\vec{K}_n$. Let $\vec{A}_n$ denote the skew-adjacency matrix of $\vec{K}_n$. Then:
\[
\vec{A}_n=\left[ \begin{array}{ccccc} 0 & 1 & 1 & \cdots & 1 \\ -1 & 0 & 1 & \cdots & 1 \\ 
\vdots & \vdots & \ddots & \vdots & \vdots \\ -1 & -1 & -1 & \cdots & 0 \end{array}\right].
\]

\begin{lemma}\label{detknlemm} The following hold for the skew-adjacency matrix $\vec{A}_n$ of $\vec{K}_n$.
\begin{enumerate}
\item For $n \in 2 \mathbb{N}$, $\vec{A}_n$ is invertible and $\det(\vec{A}_n)=1$.
\item For $n \notin 2 \mathbb{N}$, $\det(\vec{A}_n)=0$.
\end{enumerate}
\end{lemma}
\begin{proof} Fist suppose that $n \notin 2 \mathbb{N}$. It is sufficient to demonstrate that there is a nonzero vector in the $\lambda=0$ eigenspace.  Indeed, for $n$ odd, we have:
\[
\vec{v}_n=\left[1 \,\, -1 \,\, 1 \cdots -1 \,\, 1 \right]^t.
\]
By direct computation, we see that $\vec{A}_n \cdot \vec{v}_n=0$ for all odd $n$.  Hence, $\det(\vec{A}_n)=0$ for $n$ odd.

Now suppose that $n \in 2 \mathbb{N}$.  It can be easily checked that the inverse is given by:
\[
\vec{A}_n^{-1}=\left[ \begin{array}{cccccc} 0 & -1 & 1 & -1 & \cdots & -1 \\ 1 & 0 & -1 & 1 &  \cdots & 1 \\ 
\vdots & \vdots & \ddots & \vdots & \vdots & \vdots \\ 1 & -1 & 1 & -1 & \cdots & 0 \end{array}\right].
\]
Since $\vec{A}_n^{-1}$ is a matrix over $\mathbb{Z}$, it follows that $\det(\vec{A}_n)=\pm 1$ (as these are the only units in $\mathbb{Z}$). We will compute the determinant of $\vec{A}_n$ using the identity:
\[
\vec{A}_n^{-1}=\frac{1}{\det(\vec{A}_n)} \text{adj}(\vec{A}_n)=\pm \text{adj}(\vec{A}_n),
\]
where $\text{adj}(B)$ denotes the adjugate of $B$. The proof is by induction on $n$.  For $n=2$, we have $\vec{A}_2=\left[\begin{array}{cc}0 & 1 \\ -1 & 0 \end{array}\right]$ and hence $\det(\vec{A}_2)=1$.  Suppose that the proposition is true for all even numbers less than $n$. For a matrix $B$, let $B^{\hat{j}}$ denote the matrix obtained by deleting the $j$-th column of $B$. For a square matrix $B$, Let $B_{\hat{i}}^{\hat{j}}$ denote the square matrix obtained from $B$ by deleting the $i$-th row and the $j$-th column.  We expand the determinant given by the $(1,2)$-cofactor of $\vec{A}_n$:
\begin{eqnarray*}
\left|
\begin{array}{c|cccc}
-1 & 1 & 1 & \cdots & 1 \\ \hline
 -1 & & & &  \\
 \vdots & &   & \vec{A}_{n-2} & \\ \\
 -1 & & & &  \\ 
\end{array}
\right| &=& -\det(\vec{A}_{n-2})+\sum_{i=1}^{n-2} (-1)^i \left|
\begin{array}{c|ccc}
-1 &  &  &  \\
\vdots & & (\vec{A}_{n-2})^{\hat{i}} &  \\
-1 & & &   \\
\end{array}
\right|  \\
&=& -1+\sum_{i=1}^{n-2} (-1)^{i+1} \left|
\begin{array}{c|ccc}
1 &  &  &  \\
\vdots & & (\vec{A}_{n-2})^{\hat{i}} &  \\
1 & & &   \\
\end{array}
\right|\\
&=& -1+\sum_{i=1}^{n-2}(-1)^{i+1} \sum_{j=1}^{n-2} (-1)^{j-1}\left| (\vec{A}_{n-2})_{\hat{j}}^{\hat{i}} \right| \\
&=& -1+\sum_{i=1}^{n-2} \sum_{j=1}^{n-2} (-1)^{i+j} \left| (\vec{A}_{n-2})_{\hat{j}}^{\hat{i}} \right| \\
&=& -1+\sum_{i=1}^{n-2} \sum_{j=1}^{n-2} (\text{adj}(\vec{A}_{n-2}))_{ij} \\
&=& -1+\sum_{j=1}^{n-2} (-1)^j \\
&=& -1.
\end{eqnarray*}
In the third line from the bottom, we are using the induction hypothesis that $\vec{A}_{n-2}^{-1}=\text{adj}(\vec{A}_{n-2})$.  Thus, $1$ is the $(2,1)$ entry of $\text{adj}(\vec{A}_n)$.  Since $\vec{A}_n^{-1}=\pm \text{adj}(\vec{A}_n)$, and the $(2,1)$ entry of $\vec{A}_n^{-1}$ is also $1$, we must have that $\det(\vec{A}_n)=1$. This completes the proof by mathematical induction.
\end{proof}

\begin{theorem} \label{kncharpoly} The following holds for the linearly ordered complete graphs $\vec{K}_n$.
\[
\vec{P}_{\vec{K}_n}(x)=\frac{1}{2}\left((-1+x)^n+(1+x)^n\right).
\]
\end{theorem}
\begin{proof} We will again use the fact that $\vec{P}_{\vec{G}}'(x)=\sum_{j=1}^n \vec{P}_{\vec{G}-j}(x)$.  Let $\vec{G}=\vec{K}_n$.  We note that for all $j=1,\ldots,n$, $\vec{G}-j\cong \vec{K}_{n-1}$. Hence, $\vec{P}_G'(x)=n \cdot \vec{P}_{\vec{K}_{n-1}}(x)$.  Suppose firstly that $n$ is even.  It follows that $\vec{P}_{\vec{G}}^{(n-2)}(x)=\frac{n!}{2}(x^2+1)$. Recall that $\vec{P}_{\vec{K}_{t}}(0)=(-1)^t\det(\vec{A}_t)$ for all $t$. By Theorem \ref{detknlemm}, we have that $\vec{P}_{\vec{K}_{t}}(0)=1$ for $t$ even and $0$ for $t$ odd. From these observations, we conclude that:
\[
\vec{P}_{\vec{K}_n}^{(k)}(0)=\left\{\begin{array}{cc} \frac{n!}{(n-k)!} & k \le n-4, k \text{ even} \\ 0 & k \le n-3, k \text{ odd} \end{array} \right. .
\]
By Taylor's Theorem, we have that:
\begin{eqnarray*}
\vec{P}_{\vec{K}_n}(x) &=& \sum_{i=0}^n \frac{\vec{P}^{(i)}_{\vec{K}_n}(0)}{i!} x^i \\
           &=& \frac{n!}{2} \left(\frac{x^n}{n!/2}+\frac{x^{n-2}}{(n-2)!}\right)+\sum_{j=0}^{\frac{n-4}{2}} \frac{n!}{(n-2j)!(2j)!} x^{2j} \\
           &=& \sum_{j=0}^{n/2} \left(\begin{array}{c}n \\ 2j \end{array} \right) x^{2j} \\
           &=& \frac{1}{2}\left((-1+x)^n+\left(1+x\right)^n\right).
\end{eqnarray*}
The last equality follows from the Binomial Theorem. Now suppose that $n$ is odd. In this case we have that $\vec{P}_{\vec{G}}^{(n-1)}(x)=n! x$.  Using an argument similar to that of the $n \in 2 \mathbb{N}$ case, it follows that:
\[
\vec{P}_{\vec{K}_n}^{(k)}(0)=\left\{\begin{array}{cc} \frac{n!}{(n-k)!} & k \text{ odd} \\ 0 & k \text{ even} \end{array} \right. .
\]
Taylor's Theorem and the Binomial Theorem imply that:
\[
\vec{P}_{\vec{G}}(x)=\sum_{j=1}^{(n+1)/2} \left(\begin{array}{c}n \\ 2j-1 \end{array} \right) x^{2j-1}=\frac{1}{2}\left(-(1-x)^n+(1+x)^n\right).
\]
\end{proof}

\section{Application To Pretzel Knots} \label{pretzels}
Let $\mathscr{F}$ be a set (not necessarily finite) of virtual knot diagrams (in particular, we do not consider them up to Reidemeister equivalence). Let $ m \ge 1$ be a natural number.  Let $j$ be a natural number, $0 \le j \le m$. 
\newline
\newline
\centerline{
\fbox{\parbox{6in}{\underline{\textbf{Question} $(j,m)$:}  Given a diagram in $K \in \mathscr{F}$, how many ways are there to give the oriented smoothing on $m-j$ crossings and the unoriented smoothing at $j$ crossings so that the result has exactly one connected component?}}
}
\newline
\newline
We will take the convention that if $m$ is greater than the number of crossings of $K \in \mathscr{F}$, then the answer to Question $(j,m)$ is $0$.

Using our notation, an equivalent formulation of Question $(j,m)$ would be: Given a diagram $K \in \mathscr{F}$, for how many partial states $S=(S_o,S_u,S_{\emptyset})$ of $K$ is it true that $|S_u|=j$, $|S_o|=m-j$, and $\#(K|S)=1$.

Now, let $\mathscr{F}$ denote the set of diagrams of pretzel knots.  We will answer Question $(0,m)$ and $(1,m)$ for all $m$. Let $p,q,r \in \mathbb{Z}\backslash\{0\}$.  Recall that a pretzel link is a link of the form shown in Figure \ref{pretzel}, where inside the boxes we have the 2-braids $\sigma^p$, $\sigma^q$, $\sigma^r$ respectively \cite{bz}.
\begin{figure}
\[
\begin{array}{|c|c|c|} \hline
\begin{array}{c} \scalebox{.5}{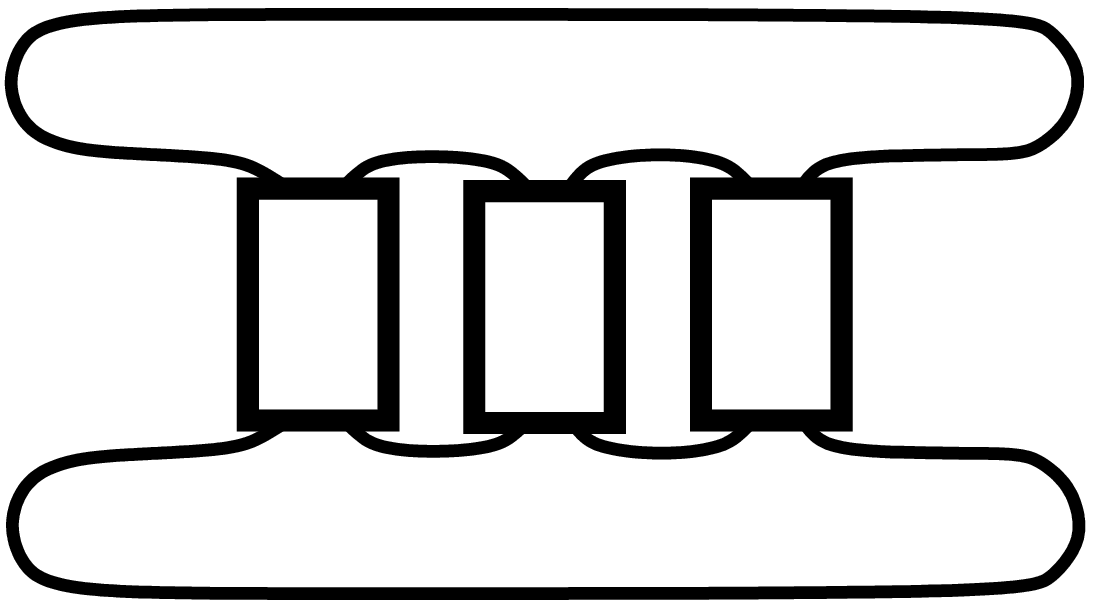}\end{array} & \begin{array}{c}\scalebox{.07}{\psfig{figure=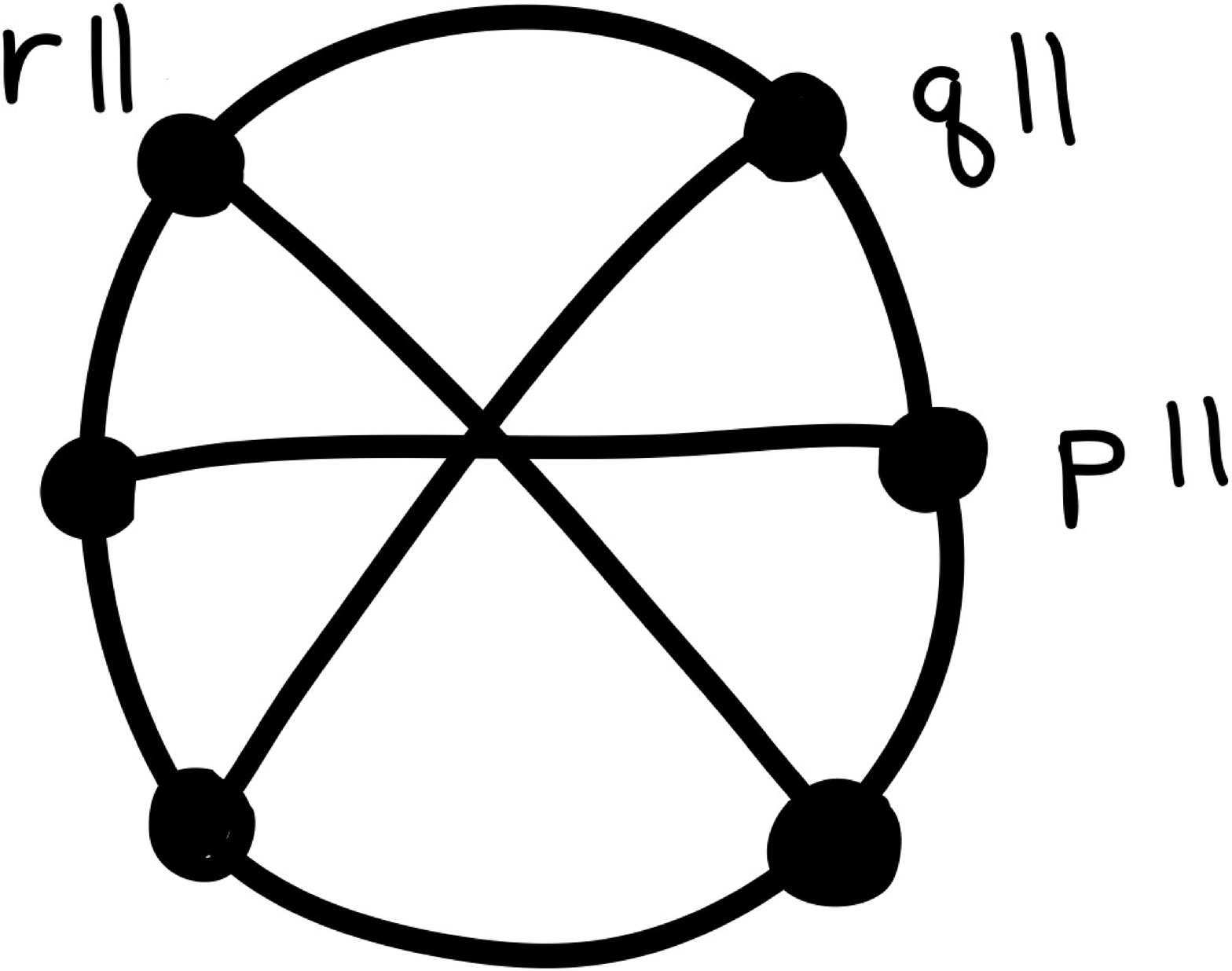}} \end{array}& \begin{array}{c}\scalebox{.05}{\psfig{figure=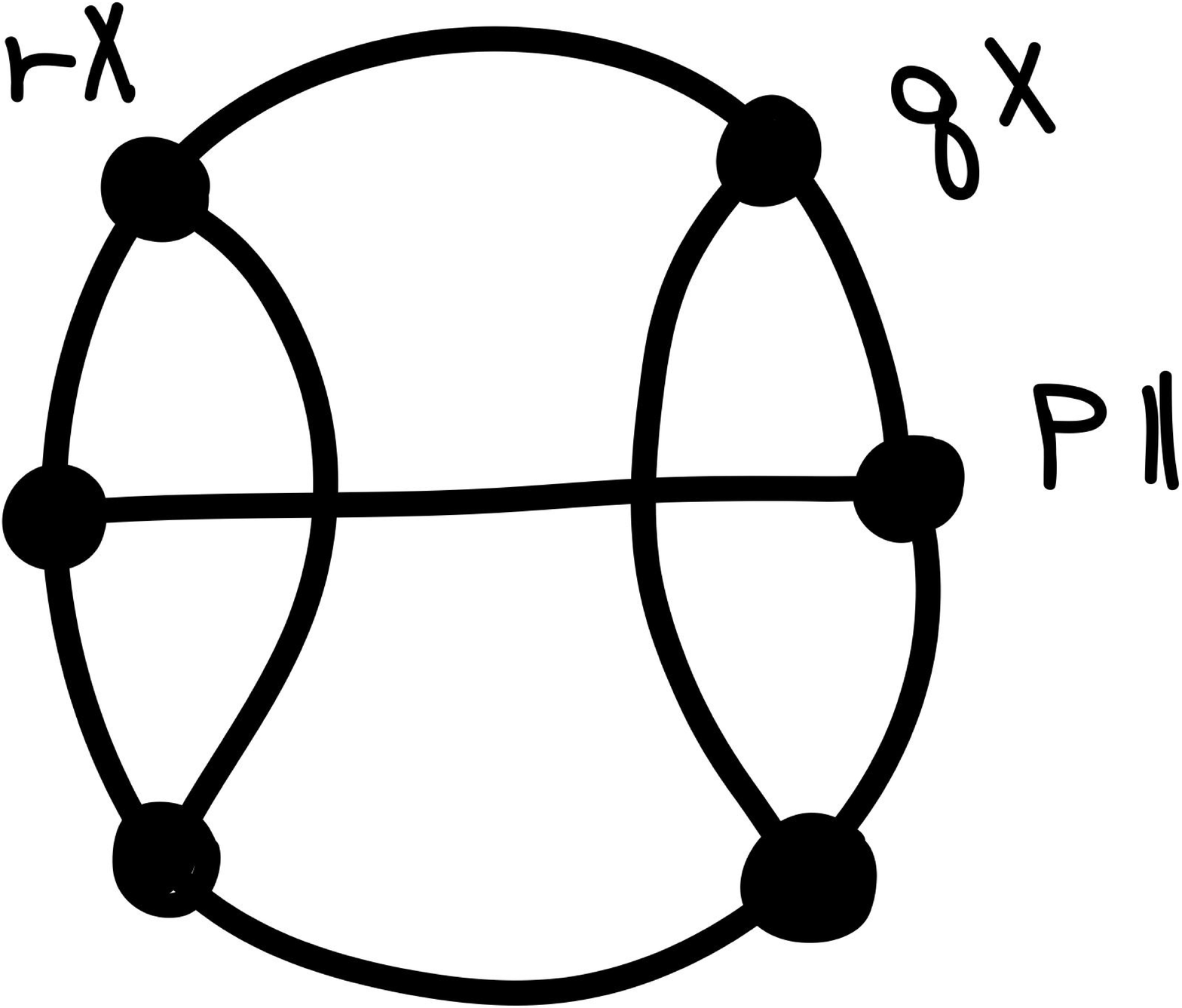}}\end{array}\\
p,q,r \in \mathbb{Z}\backslash\{0\} & p,q,r \text{ all odd} & p \text{ even}, q,r \text{ odd} \\ \hline
\end{array}
\]
\caption{Two cases of pretzel knots.} \label{pretzel}
\end{figure}
It is easy to check that the $L(p,q,r)$ pretzel link is a pretzel knot if and only if at most one of $p$, $q$ and $r$ is even. When all of $p$, $q$ and $r$ are odd, then the Gauss diagram resembles the middle of Figure \ref{pretzel}.  Here, the chord represents $p$, $q$, or $r$ parallel chords.  When one of them is even, say $p$, the Gauss diagram resembles the right hand side of Figure \ref{pretzel}.  The chord of degree two represents $p$ parallel chords.  The chords of degree $1$ represent $q$ and $r$ chords whose intersection graphs are the complete graphs $K_q$ and $K_r$, respectively.  

First we prove a number of results which will be of use to answer Question $(0,m)$ and $(1,m)$. Note that instead of writing out the characteristic polynomials explicitly, we will just write out their values at zero and their derivatives at zero. In the figure below, the regions labelled with $\alpha$ (or $\beta$) have a total of $\alpha$ chords (respectively, $\beta$ chords) with at least one endpoint in them such that every chord in the region intersects every other chord in the region.  Black chords having no endpoints contained in an $\alpha$ region or $\beta$ region represent single chords. If a black chord passes through a region, it intersects every chord with an endpoint in that region.

\begin{figure}[h]
\begin{tabular}{cc}
$D_{\ref{kalphakbeta_lemma}}=\begin{array}{c}\scalebox{1}{\psfig{figure=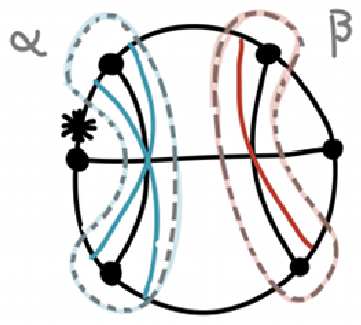}}\end{array}$, & $D_{\ref{fouralphabeta}}=\begin{array}{c}\scalebox{1}{\psfig{figure=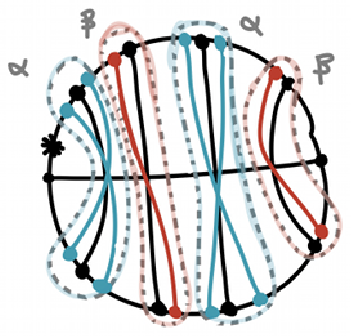}}\end{array}$, \\ 
$D_{\ref{parralemm}}=\begin{array}{c}\scalebox{1}{\psfig{figure=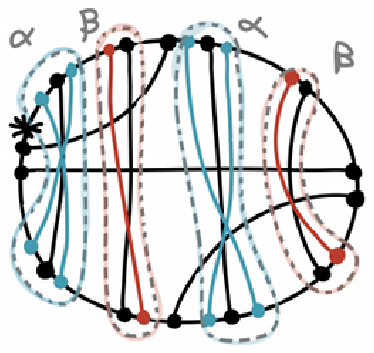}}\end{array}$,
& $D_{\ref{bsmoothcrosslemma}}=\begin{array}{c}\scalebox{1}{\psfig{figure=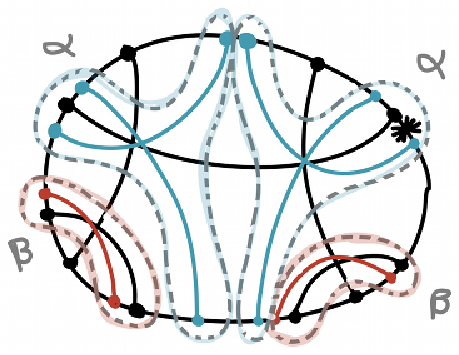}}\end{array}$
\end{tabular}
\caption{Diagrams considered in Lemmas \ref{kalphakbeta_lemma}, \ref{fouralphabeta}, \ref{parralemm}, and \ref{bsmoothcrosslemma}.} \label{lemmafig}
\end{figure}

\begin{lemma} \label{kalphakbeta_lemma} Let $\vec{G}_{\ref{kalphakbeta_lemma}}$ be the linearly ordered graph associated to $D_{\ref{kalphakbeta_lemma}}$ in Figure \ref{lemmafig}. Then:
\begin{eqnarray*}
\vec{P}_{\vec{G}_{\ref{kalphakbeta_lemma}}}(0) &=& \frac{1}{2}\left(1-(-1)^{\alpha+\beta}\right), \\
\vec{P}_{\vec{G}_{\ref{kalphakbeta_lemma}}}'(0) &=& \frac{1}{4}\left(1+(-1)^{\alpha}+(-1)^{\beta}+2\alpha+2\beta+(-1)^{\alpha+\beta}(1+2\alpha+2\beta)\right) .\\
\end{eqnarray*}
\end{lemma}
\begin{proof} Consider the linearly ordered graphs $\vec{K}_{\alpha+1}$ and $\vec{K}_{\beta+1}$.  We form the coalescence over the vertices labelled $1$ in each graph. Then $\vec{G}_{\ref{kalphakbeta_lemma}}$ is $\vec{K}_{\alpha+1}\cdot \vec{K}_{\beta+1}$.  By Lemma \ref{kncharpoly}, the polynomials $\vec{P}_{\vec{K}_{\alpha+1}}(x)$ and $\vec{P}_{\vec{K}_{\beta+1}}(x)$ are known. Note that for any vertex $j \in V(\vec{K}_{\beta})$, the promotion of $j$ does not change the linearly ordered graph: $\vec{K}_{\beta+1}$ is identical to $\vec{K}_{\beta+1}\leftrightarrow j$.  Then by Theorem \ref{coalthm}, we can determine $\vec{P}_{\vec{K}_{\alpha+1}\cdot \vec{K}_{\beta+1}}(x)$. From this we can determine the polynomial and its derivative at zero (as above, simplified using \emph{Mathematica}).
\end{proof}
\begin{lemma} \label{fouralphabeta} Let $\vec{G}_{\ref{fouralphabeta}}$ be the linearly ordered graph associated to $D_{\ref{fouralphabeta}}$ in Figure \ref{lemmafig}.  Then:
\begin{eqnarray*}
\vec{P}_{\vec{G}_{\ref{fouralphabeta}}}(0) &=& \frac{1}{8}\left(1+(-1)^{\alpha}\right)\left(1+(-1)^{\beta}\right)\\
& \cdot & \left(\left(1+(-1)^{\alpha+1}\right)
\left(1+(-1)^{\beta}\right)+\left(1+(-1)^{\alpha}\right)\left(1+(-1)^{\beta+1}\right)\right) ,\\
\vec{P}_{\vec{G}_{\ref{fouralphabeta}}}'(0) &=& \frac{1}{16}\left(8\alpha \left(1+(-1)^{\beta}\right)\left(1+(-1)^{2\alpha+\beta}\right)+(1+(-1)^{\alpha}) \right. \\
&\cdot & \left.\left( \left(1+(-1)^{\alpha}\right)(1+(-1)^{\beta})^2+8\beta\left(1+(-1)^{\alpha+2\beta}\right)\right)\right).
\end{eqnarray*}
\end{lemma}
\begin{proof} Consider the linearly ordered graphs $\vec{K}_{\alpha+1}$ and $\vec{K}_{\beta+1}$.  First form the coalescence over the vertices labelled $1$: $\vec{K}_{\alpha}\cdot \vec{K}_{\beta}$.  The new vertex, which is labelled 1, is adjacent to a copy of $\vec{K}_{\alpha}$ and a copy of $\vec{K}_{\beta}$. Now take two copies of this $\vec{K}_{\alpha+1} \cdot \vec{K}_{\beta+1}$ and form the coalescence over the vertices labelled $1$.  It follows that $\vec{G}_{\ref{fouralphabeta}}$ is given by $(\vec{K}_{\alpha+1} \cdot \vec{K}_{\beta+1}) \cdot (\vec{K}_{\alpha+1} \cdot \vec{K}_{\beta+1})$. Since we are always coalescing over the vertices labelled 1, promotion does not affect the characteristic polynomial. We have a closed form for $\vec{P}_{\vec{K}_{\alpha+1} \cdot \vec{K}_{\beta+1}}(x)$ from the previous lemma.  On the other hand, deleting the vertex $1$ in $(\vec{K}_{\alpha+1}\cdot \vec{K}_{\beta+1})$ gives the disjoint union $\vec{K}_{\alpha} \sqcup \vec{K}_{\beta}$.  Hence, $\vec{P}_{\vec{G}_{\ref{fouralphabeta}}}(x)$ can be computed from Theorem \ref{coalthm}.  The formulas for the polynomial and its derivative at $0$ follow.     
\end{proof}
\begin{lemma} \label{parralemm} Let $\vec{G}_{\ref{parralemm}}$ be the linearly ordered graph associated to $D_{\ref{parralemm}}$ in Figure \ref{lemmafig}.  Then:
\begin{eqnarray*}
\vec{P}_{\vec{G}_{\ref{parralemm}}}(0) &=& 0,\\
\vec{P}_{\vec{G}_{\ref{parralemm}}}'(0) &=& \frac{3}{4}\left(-1+(-1)^{\alpha+\beta} \right)^2.\\
\end{eqnarray*}
\end{lemma}
\begin{proof} First consider $\vec{H}_1=K_2 \vec{\nabla}(\vec{K}_{\alpha} \sqcup K_{\beta})$. Then by Theorem \ref{jointhm}, $\vec{P}_{\vec{H}_1}(x)$ is determined solely by $\vec{P}_{\vec{K}_{\alpha}}(x)$, $\vec{P}_{\vec{K}_{\beta}}(x)$,$\vec{P}_{\vec{K}_1 \vec{\nabla} (\vec{K}_{\alpha} \sqcup \vec{K}_{\beta})}(x)$, and $\vec{P}_{\vec{K}_3}(x)$. It is easy to see that $K_1 \vec{\nabla} (\vec{K}_{\alpha} \sqcup \vec{K}_{\beta})$ is the same as $\vec{K}_{\alpha+1} \cdot \vec{K}_{\beta+1}$, as computed in Lemma \ref{kalphakbeta_lemma}.

Let $\vec{H}_2$ denote the linearly ordered graph obtained by deleting the blue and red subsets of arrows on the right hand side of $D_{\ref{parralemm}}$ and the resulting isolated chord on the bottom right. Then $\vec{H}_2$ contains two subsequent chords with the same adjacency.  Then $\vec{H}_2$ may be obtained from $\vec{H}_1$ by deleting the edge between the vertices labelled $1$ and $2$. If follows from Corollary \ref{edgdelcorr} that:
\[
\vec{P}_{\vec{H}_2}(x)=\vec{P}_{\vec{H}_1}(x)-\vec{P}_{\vec{K}_{\alpha}}(x)\vec{P}_{\vec{K}_{\beta}}(x).
\]
Finally, we can form the linearly ordered graph of $D_{\ref{parralemm}}$ as $\vec{H}_2 \cdot \vec{H}_2$, where the coalescence is taken over the vertex labelled $2$ in the first graph and the vertex labelled $1$ in the second graph. Once again, promotion does not affect the graph.  Hence, $\vec{P}_{\vec{G}_{\ref{parralemm}}}(x)$ is determined from Theorem \ref{coalthm}. The polynomial and derivative at zero follow. 
\end{proof}

\begin{lemma}\label{bsmoothcrosslemma} Let $\vec{G}_{\ref{bsmoothcrosslemma}}$ be the linearly ordered graph associated to $D_{\ref{bsmoothcrosslemma}}$ in Figure \ref{lemmafig}.  Then:
\begin{eqnarray*}
\vec{P}_{\vec{G}_{\ref{bsmoothcrosslemma}}}(0) &=& 0 ,\\
\vec{P}_{\vec{G}_{\ref{bsmoothcrosslemma}}}'(0) &=& \frac{1}{4}\left(-1+(-1)^{\beta}+(-1)^{\alpha}+2(-1)^{\alpha+\beta}\right.\\
&+&\left.(-1)^{2\alpha+\beta}+(-1)^{\alpha+2\beta}+8\alpha+8\beta\right).
\end{eqnarray*}
\end{lemma}
\begin{proof} Start with $\vec{H}_1=\vec{K}_{\alpha+1}\cdot \vec{K}_{\beta+1}$, where the coalescence is taken over the vertex labelled $1$ in $\vec{K}_{\beta+1}$ and some vertex other than the vertex labelled $1$ in $K_{\alpha+1}$. Then take two copies of $\vec{H}_1$ and form $\vec{H}_1 \cdot \vec{H}_1$ over the vertices labelled $1$. Recall that $\bar{H}$ denotes the intersection graph of a the mirror image of a diagram.  Recall that for linearly ordered graphs obtained in this way, we have from Theorem \ref{mirrorthm} that $\vec{P}_{\bar{H}}(x)=\vec{P}_{\vec{H}}(x)$.  It follows that $\vec{P}_{\vec{G}_{\ref{bsmoothcrosslemma}}}(x)=\vec{P}_{\vec{H}_1\cdot \vec{H_1}}(x)$.  This latter polynomial can be computed easily from two applications of Theorem \ref{coalthm}. The statements about the polynomial and its derivative at zero follow from this computation.   
\end{proof}

We are now poised to answer Questions $(0,m)$ and $(1,m)$.  Let $p,q,r \in \mathbb{Z}\backslash\{0\}$ and $P=|p|$, $Q=|q|$, $R=|r|$. Let $m \in \mathbb{N}$. Let $N_0(p,q,r,m)$ be the number of ways that $m$ crossings can be chosen from the pretzel knot $L(p,q,r)$ such that the oriented smoothing at those $m$ crossings has one component. Let $N_1(p,q,r,m)$ be the number of ways that $m$ crossings can be chosen from the pretzel knot $L(p,q,r)$ such that exactly one is an unoriented smoothing, exactly $m-1$ are oriented smoothings, and the result has exactly one component. In other words, $N_j$ is the answer to Question $(j,m)$ for the parameters $p,q,r,m$.
\begin{theorem} \label{pqrodd} Suppose that $p,q,r$ are all odd.
\begin{eqnarray*}
N_0(p,q,r,m) &=& \left\{\begin{array}{cc} 0 & m \ne 2 \\ PQ+QR+RS & m=2 \end{array} \right.,\\
N_1(p,q,r,m) &=& \left\{\begin{array}{cc} 0 & m>3 \\ P+Q+R & m=1 \\ 2(PQ+QR+RS) & m=2 \\ 3PQR + 2{P \choose 2}(Q+R)+2{Q \choose 2}(P+R)+2{R \choose 2}(P+Q) & m=3 \\ \end{array} \right. .\\
\end{eqnarray*}
\end{theorem}
\begin{proof} Note that if a pair of parallel chords is chosen, and both carry the oriented smoothing, then the number of components is automatically at least two, regardless of any other choices made.

Consider first the formula for $N_0$. If $m=1$, the the oriented smoothing will give two components. If $m>3$, then at least one pair of parallel chords must be chosen in Figure \ref{pretzel}. It follows that $N_0=0$ in this case. Similarly, if $m=3$ and two parallel chords are chosen, we get $N_0=0$.  If $m=3$ and all chords are intersecting, we can check that $N_0=0$. If $m=2$, the number of components will be one exactly one only when a pair of intersecting chords is chosen.  This gives the formula as above.

Consider now the formula for $N_1$.  If $m=1$, we give the unoriented smoothing at exactly one crossing.  This will give one component. If $m=2$, the number of components will be one only when a pair of intersecting chords is chosen. One of the chords will have the unoriented smoothing and the other will have the oriented smoothing. Hence, the number of ways is $2(PQ+QR+RS)$. 

For $m=3$, the number of components will be one when (1) none of the chords are parallel or (2) when two parallel chords are chosen, one of which has the unoriented smoothing, and the third chord intersects both parallel chords.   The contribution in (1) is $3PQR$.  The contribution in (2) is $2{P \choose 2}(Q+R)+2{Q \choose 2}(P+R)+2{R \choose 2}(P+Q)$. 

When $m>3$, we must choose at least two parallel chords. If the unoriented smoothing is not amongst a set of parallel chords, then the number of components is at least 2. Therefore, we have that $m \le 4$, Suppose that $m=4$.  Then the only possibility is that two chords are parallel and that two chords intersect all three of the other chords. Moreover, we must have that the unoriented smoothing is chosen to be one of the two parallel arrows.  By drawing such a chord diagram and checking, we see that the number of components is two.  The formula above follows.
\end{proof}
\begin{theorem} Suppose that one of $p$, $q$ and $r$ is even (say, $p$). Then we have:
\begin{eqnarray*}
\underline{m = 1:} \,\,\,\,\,\,\,\,\,\,\,\,\,\,\,\, & & \\
N_0(p,q,r,1) &=& 0 ,\\
N_1(p,q,r,1) &=& P+Q+R ,\\
\underline{m \ge 2, \text{ even}:} & & \\
N_0(p,q,r,m) &=& \sum_{k=0}^{m/2} {Q \choose 2k }{R \choose m-2k}+P \cdot {Q \choose 2k}{R \choose m-1-2k}+P\cdot{R \choose 2k}{Q \choose m-1-2k} ,\\
N_1(p,q,r,m) &=& P \cdot \sum_{k=0}^{(m-2)/2} {Q \choose 2k}{R \choose m-1-2k}+{Q \choose m-1-2k}{R \choose 2k}\\
&+& P Q\sum_{k=0}^{m-2} { Q-1 \choose k}{R \choose m-2-k}+Q \sum_{k=0}^{\lfloor \frac{ m-1}{2} \rfloor}{ R \choose 2k }{Q-1 \choose m-1-2k}\\
&+& 
P R\sum_{k=0}^{m-2} { R-1 \choose k}{Q \choose m-2-k}+R \sum_{k=0}^{\lfloor \frac{ m-1}{2} \rfloor}{ Q \choose 2k }{R-1 \choose m-1-2k} ,\\
\underline{m \ge 3, \text{ odd}:} & & \\
N_0(p,q,r,m) &=& 0,\\
N_1(p,q,r,m) &=& P \cdot \sum_{k=0}^{(m-1)/2} {Q \choose 2k}{R \choose m-1-2k}  \\
             &+& P \cdot \sum_{k=0}^{(m-3)/2} {Q \choose 2k}{R \choose m-2-2k}+{Q \choose m-2-2k}{R \choose 2k} \\
             &+& P Q\sum_{k=0}^{m-2} { Q-1 \choose k}{R \choose m-2-k}+Q \sum_{k=0}^{(m-1)/2}{ R \choose 2k }{Q-1 \choose m-1-2k}\\
&+& 
P R\sum_{k=0}^{m-2} { R-1 \choose k}{Q \choose m-2-k}+R \sum_{k=0}^{(m-1)/2}{ Q \choose 2k }{R-1 \choose m-1-2k}. \\
\end{eqnarray*}
\end{theorem}
\begin{proof} The formula for $N_0$ follows immediately from Lemma \ref{kalphakbeta_lemma}.

Let $S=(S_o,S_u,S_{\emptyset})$ be a partial smoothing with $|S_u|=1$. For $N_1$, first choose a chord $a$ to be the unoriented smoothing. Then $a$ may be amongst the $p$ chords, the $q$ chords or the $r$ chords. Suppose first that $a$ is amongst the $p$ parallel chords. Immediately we see that if more than $1$ of the remaining $m-1$ chords are chosen from the amongst the $p-1$ chords parallel to $a$, then you get more than one component.  Hence, you can choose $1$ or $0$ from these chords. 

Now take the double cover $\Sigma_D^2(S)$ defined in Section 2.4 and contract along one of the two bands corresponding to $a$. This gives a Gauss diagram $D_a^f(S)$ (or $D_a^s(S)$) as previously described. The question for $D$ translates to the question of a no-unoriented smoothing on $D_a^f(S)$ for a choice of $a$ and $2(m-1)$ other chords which gives \emph{exactly two components}. Note that the chords other than $a$ in $D_a^f(S)$ are chosen in pairs via the inverse image of the double cover. Let $\alpha$ denote the number of chords chosen from amongst the $q$ chords and $\beta$ the number of chords chosen from amongst the $r$ chords. 

Suppose that $0$ of the chords parallel to $a$ are chosen for the oriented smoothing. Then $D_a^f(S)$ is $D_{\ref{fouralphabeta}}$ of Figure \ref{lemmafig}. By Lemma \ref{fouralphabeta}, $\vec{P}_{\vec{G}_{\ref{fouralphabeta}}}(0)=0$ for all $\alpha,\beta \ge 1$ and $\vec{P}_{\vec{G}_{\ref{fouralphabeta}}}'(0)\ne 0$ when $\alpha,\beta$ are not both odd. If $\alpha=\beta=0$, then $m=1$ and the answer is $P+Q+R$.  If $\alpha=0$, $\beta \ne 0$ or $\alpha\ne 0$, $\beta=0$, then by Lemma \ref{kalphakbeta_lemma}, $D_a^f(S)$ has two components.  Thus the contribution to $N_1$ from this case for $m\ge 2$ is:
\begin{eqnarray*}
\underline{m \text{ odd}} &: &P \cdot \sum_{k=0}^{(m-1)/2} {Q \choose 2k}{R \choose m-1-2k},  \\
\underline{m \text{ even}} & : & P \cdot \sum_{k=0}^{(m-2)/2} {Q \choose 2k}{R \choose m-1-2k}+{Q \choose m-1-2k}{R \choose 2k}. 
\end{eqnarray*}

Now suppose that one of the chords parallel to $a$ is chosen for an oriented smoothing.  Then $D_a^f(S)$ is $D_{\ref{parralemm}}$ in Figure \ref{lemmafig}. In this case we need to choose $\alpha$, $\beta$ so that $\alpha+\beta=m-2$. By Lemma \ref{parralemm}, $\vec{P}_{\vec{G}_{\ref{parralemm}}}(0)=0$ whenever $\alpha \ge 1$, $\beta \ge 0$ or $\beta \ge 1, \alpha \ge 0$ and $\vec{P}_{\vec{G}_{\ref{parralemm}}}'(0)\ne 0$ whenever $\alpha$ and $\beta$ have opposite parity. If $\alpha=\beta=0$, then $m=2$ and $D_{\ref{parralemm}}$ has $4$ components. Thus the contribution to $N_1$ is:
\begin{eqnarray*}
\underline{m \text{ odd}} &: & P \cdot \sum_{k=0}^{(m-3)/2} {Q \choose 2k}{R \choose m-2-2k}+{Q \choose m-2-2k}{R \choose 2k} ,\\
\underline{m \text{ even}}&: &  0 .
\end{eqnarray*}

Suppose that the unoriented smoothing is chosen from amongst the $q$ chords. If more than one of the $p$ chords is chosen for the oriented smoothings, then there are more than two components. So either zero of the $p$ chords are chosen or one of the $p$ chords is chosen.   If one of the $p$ is chosen, $D_a^f(S)$ is the diagram $D_{\ref{bsmoothcrosslemma}}$ in Figure \ref{lemmafig}. Note that we must have $\alpha \ge 1$.  By Lemma \ref{bsmoothcrosslemma}, $\vec{P}_{\vec{G}_{\ref{bsmoothcrosslemma}}}(0)=0$ for all $\alpha,\beta \ge 1$ and  $\vec{P}_{\vec{G}_{\ref{bsmoothcrosslemma}}}'(0) \ne 0$ for all $\alpha, \beta \ge 1$. If $\beta=0$, then Lemma \ref{kalphakbeta_lemma} implies that there are two components for every choice of $\alpha$.  Thus the contribution to $N_1$ for this case is:
\[
P Q\sum_{k=0}^{m-2} { Q-1 \choose k}{R \choose m-2-k}.
\]

Now suppose that none of the $p$ chords is chosen. Then we have $\alpha+\beta=m$, where $\alpha\ge 1$.  Then by moving the basepoint if necessary (this does not affect the number of boundary components), the skew characteristic polynomial of the linearly order graph is given $(\vec{P}_{\vec{K}_{\beta}}(x))^2 \cdot \vec{P}_{\vec{K}_{\alpha} \cdot \vec{K}_{\alpha}}(x)$ for all $\alpha\ge 1$, $\beta \ge 0$. This will give two components exactly when $\beta$ is even and $\alpha$ is any number greater than or equal to $1$. Then the contribution to $N_1$ is:
\[
Q \sum_{k=0}^{\lfloor \frac{ m-1}{2} \rfloor}{ R \choose 2k }{Q-1 \choose m-1-2k}.
\]
Similarly, the contribution to $N_1$ if the unoriented smoothing is chosen from the $r$ chords is given by:
\[
P R\sum_{k=0}^{m-2} { R-1 \choose k}{Q \choose m-2-k}+R \sum_{k=0}^{\lfloor \frac{ m-1}{2} \rfloor}{ Q \choose 2k }{R-1 \choose m-1-2k}.
\]
Adding all of the contributions together and accounting for the parity of $m$, we obtain the indicated formula.  This completes the proof. 

\end{proof}

\section{Some Problems and Questions} \label{problems}
We conclude with a list of questions about the skew-spectra of virtual knots.  The author has spent some time investigating them, but has not made significant progress toward their resolution. It is hoped that this list will inspire a more sophisticated investigation.
\begin{enumerate}
\item For each partial state $S$ of a Gauss diagram $D$, is the skew-spectrum of $S$ determined by the skew-spectrum of the all-oriented state?
\item For each partial state $S$ of a Gauss diagram $D$ having at least one unoriented smoothing, can the linearly ordered graph of the double cover be determined from graph operations on the linearly ordered graph $\vec{G}_D$?
\item For Gauss diagrams associated to knot theory relations (e.g Reidemeister moves, six-term relations, four-term relations), what relations appear in the skew-spectra of their linearly ordered graphs?
\item What is the answer to Question $(j,m)$ for other infinite parametrized families of virtual knots?
\item Answers to Question $(j,m)$ will be linear combinations of generalized hypergeometric functions.  Also, we know that knot polynomials and finite-type invariants count certain subdiagrams of Gauss diagrams.  Are there any hypergeometric identities which can be derived using these counting principles, Reidemeister equivalence, and mutations? 
\item (Manturov) Which combinatorial formulae can be defined on intersection graphs (or on graph spectra) and which of those can be integrated as finite-type invariants of free knots?
\end{enumerate}
\bibliographystyle{plain}
\bibliography{bib_comm}
\end{document}